\documentclass{amsart}

\usepackage{amsfonts, latexsym, amsmath, amsthm, amscd, eucal, amssymb, bbm, mathtools,enumerate} 
\usepackage[utf8]{inputenc}
\usepackage[normalem]{ulem}
\usepackage{hyperref}

\usepackage{faktor}
\usepackage{tikz}
\usetikzlibrary{calc,backgrounds,positioning,shadows,fadings}
\def\Xint#1{\mathchoice
{\XXint\displaystyle\textstyle{#1}}%
{\XXint\textstyle\scriptstyle{#1}}%
{\XXint\scriptstyle\scriptscriptstyle{#1}}%
{\XXint\scriptscriptstyle\scriptscriptstyle{#1}}%
\!\int}
\def\XXint#1#2#3{{\setbox0=\hbox{$#1{#2#3}{\int}$ }
\vcenter{\hbox{$#2#3$ }}\kern-.6\wd0}}

\def\dashint{\Xint-}

\newtheorem{theorem}{Theorem}[section]
\newtheorem{lemma}[theorem]{Lemma}
\newtheorem{proposition}[theorem]{Proposition}
\newtheorem{corollary}[theorem]{Corollary}

\newtheorem{assumption}[theorem]{Assumption}
\newtheorem{definition}[theorem]{Definition}

\newtheorem{remark}[theorem]{Remark}
\numberwithin{equation}{section}

\def\cC{\mathcal C}

\def\cF{\mathcal F}
\def\cG{\mathcal G}

\def\cK{\mathcal K}
\def\cL{\mathcal L}

\def\bC{\mathbb C}

\def\bE{\mathbb E}

\def\bN{\mathbb N}

\def\bP{\mathbb P}

\def\bR{\mathbb R}
\def\bS{\mathbb S}

\def\me{\mathsf{e}}
\def\mf{\mathsf{f}}
\def\mE{\mathsf{E}}
\def\mG{\mathsf{G}}
\def\mj{\mathsf{j}}

\def\mv{\mathsf{v}}
\def\mV{\mathsf{V}}
\def\mw{\mathsf{w}}

\def\uno{\mathbbm{1}}

\def\ie{\emph{i.e.\/}}
\def\eg{\emph{e.g.\/}}

\DeclareMathOperator{\rg}{rg}

\title[Random evolution equations]{Random evolution equations:\\ well-posedness, asymptotics, and applications to graphs}
\author{Stefano Bonaccorsi} 
\address{S.\ B.\ Dipartimento di Matematica, Universit\`a di Trento \\ via Sommarive 14, 38123 Povo TN, Italy \\ E-mail: \tt stefano.bonaccorsi@unitn.it}

\author{Francesca Cottini} 
\address{F.\ C.\ Dipartimento di Matematica e Applicazioni, Universit\`a di Milano Bicocca \\ via Roberto Cozzi 55,
	20125 Milano, Italy \\
	E-mail: \tt f.cottini2@campus.unimib.it}

\author{Delio Mugnolo}
\address{D.\ M.\ Lehrgebiet Analysis, Fakultät Mathematik und Informatik\\
	FernUniversität in Hagen, 58084 Hagen, Germany \\ E-mail: \tt delio.mugnolo@fernuni-hagen.de}

\usepackage{xcolor}
\definecolor{vio}{RGB}{146,043,062}


\thanks{The first author  gratefully acknowledge support by the Italian MIUR-PRIN 2015 ``Deterministic and stochastic evolution equations'' (Grant No. 2015233N54).
The third author was partially supported by the Deutsche Forschungsgemeinschaft (Grant 397230547). }

\begin{document}

\begin{abstract}
We study diffusion-type equations supported on structures that are randomly varying in time. After settling the issue of well-posedness, we focus on the asymptotic behavior of solutions: our main result gives sufficient conditions for pathwise convergence in norm of the (random) propagator towards a (deterministic) steady state. We apply our findings in two environments with randomly evolving features: ensembles of difference operators on combinatorial graphs, or else of differential operators on metric graphs.
\end{abstract}

\maketitle

\section{Introduction}

Randomly switching dynamical systems stand in between deterministic evolution equations (where the dynamics of the system is prescribed and completely known \textit{a priori}) and stochastic differential equations, where the dynamics
is perturbed by the introduction of noise. 

Such systems are described by a continuous component, which follows a (deterministic) {evolution} driven by an operator $A_j$ which is selected among a class of operators $\cC = \{A_1, \dots, A_n\}$ by a discrete jump process.

These problems are related to a large  -- but somehow disjoint -- literature, which treats piecewise deterministic Markov processes \cite{Bakhtin2012, Cloez2015, Malrieu2015}, switched dynamical systems \cite{Benaim2012}, products of random matrices \cite{Gurvits1995}, random walk in random environment \cite{Zei02, Zei04} with applications in biology \cite{Bressloff2017}, physics \cite{Buceta2002} or finance \cite{Yin2010}, for instance.

In the present paper, we study the asymptotic behavior of a class of random evolution problems that may be relevant in some applications.
Our main result (Theorem~\ref{t1} below) states that the system consisting of a random switching between parabolic evolution equations driven by contractive, self-adjoint, immediately compact semigroups converges towards an orthogonal projector provided the process spends enough time at each state: we refer to Section~\ref{sez:setting} for the theorem's formulation and Section~\ref{sez:proofs} for its proof.
As a motivation to our study, we provide in this section an example concerning the dynamics of the discrete heat equation on a system of random varying graphs.
This example will be further analyzed in Section \ref{sez3}, which is devoted to the study of combinatorial graphs: there we discuss some further examples which relate our results to the existing literature.
Finally, Section \ref{sez5} is devoted to an application of Theorem \ref{t1} to a randomly switching evolution system on {\em metric graphs}.
This section takes advantage of a novel formal definition of metric graphs (\cite{Mug19}) which can be exploited to verify the assumptions of our construction.

\subsection*{A motivating example}
Let $\mG_1, \ldots, \mG_N$ be a family of simple (\ie, with no loops or multiple edges) but not necessarily connected graphs on a fixed set of vertices $\mV$ with cardinality $|\mV|$.
We consider the function space defined as the complex, finite-dimensional Hilbert space $\bC^\mV \equiv \{u: \mV \to \bC\}$.

On every graph $\mG_k$ we introduce the graph Laplacian $\cL_k$ (for a formal definition, see Section \ref{sez3}), which (under our convention on the sign) is negative semi-definite
and whose eigenvalue $\lambda_1 =0$ has multiplicity equal to the number of connected components in $\mG_k$.
The corresponding eigenspace is spanned by the collection of indicator functions on each connected component. 
In particular, if $\mG_k$ is connected, then $\ker \cL_k = \langle \uno \rangle$ is the space of constant functions on the vertices.

It is known that the solution of the Cauchy problem
\[
\left\{
\begin{aligned}
\tfrac{d}{dt} u(t, \mv) &= \cL_k u(t,\mv),\qquad && \mv \in \mV,\ t\ge 0,
\\
u(0,\mv) &= f(\mv), && \mv \in \mV,
\end{aligned}
\right.
\]
can be expressed in the form
\[
u(t,x) = e^{t \cL_k}f(x)
\]
and in the limit for $t \to \infty$, it converges to the projector ${P_k}f$ onto the null space $\ker \cL_k$, where the projection equals the average of $f$ on each connected component of $\mG_k$.

Let us introduce a (random) mechanism of switching the graphs over time.
In other words, fixed a probability space $(\Omega, \cF, \bP)$, we assume that the evolution is lead by an operator $\cL_{X_k}$ (randomly selected from the set $\{ \cL_1, \dots, \cL_N\}$ according to some Markov chain $\{X_k,\ k \ge 0\}$) during the (random) time interval $[T_k,T_{k+1})$
\begin{equation}
\label{e1}
\left\{
\begin{aligned}
\tfrac{du}{dt} (t, \mv) &= \cL_{X_k} u(t, \mv), \qquad && \mv \in \mV,\ t \in [T_k, T_{k+1}), \ k\in\mathbb N,
\\
u(0, \mv) &= f(\mv), \qquad &&\mv \in \mV.
\end{aligned}
\right.
\end{equation}
We can associate with \eqref{e1} the random propagator
\begin{equation}\label{eq:randomprop}
S(t) = e^{(t - T_n)\cL_{X_n}} \prod_{k=0}^{n-1} e^{(T_{k+1} - T_k) \cL_{X_k}}, \qquad t \in [T_n, T_{n+1}),\ n\in \mathbb N,
\end{equation}
which maps each initial data $f\in \bC^\mV$ into the solution $u(t)$ of~\eqref{e1} at time $t$. This settles the issue of well-posedness of~\eqref{e1}. The main question we are going to address in this paper is however the following:
\begin{itemize}
\item[({\bf P})] Does the random propagator $(S(t))_{t\ge 0}$ converge? towards which limit?
\end{itemize}

The asymptotic behavior of a random propagator $(S(t))_{t\ge 0}$ associated with problem \eqref{e1} has not been studied in a general setting. 
Some results are known for finite-dimensional time-discrete dynamical systems, where the random propagator $(S(T_n))_{n\in\mathbb N}$ defined likewise is  a \textit{product of random matrices} (PRM for short):  this theory dates back to the 1960s, see \eg\ Furstenberg \cite{FurKes60}.

Our main result Theorem \ref{t1} requires an analysis of the null spaces of the operators ${\cL_k}$, $k=1, \dots, N$.
Notice that even when $\dim \ker\cL_k$ is constant for all $k$, there is no reason why $K = \bigcap_{j=1}^N \ker\cL_j$ should have the same dimension; describing the orthogonal projector onto $K$ is therefore, in general, no easy task. 
Coming back to our motivating example of graphs, we observe in Section~\ref{sez3} that $K$ can be explicitly described in terms of the null space of a new operator $A$ that is related to the Laplacians on the graphs $\mG_1,\ldots,\mG_n$ but acts on a different class of functions.  
The key point here is that in doing so we can relate the long time behaviors of a  Cauchy problem with random coefficients with that of an associated (deterministic) Cauchy problem supported on a different ``union'' structure { -- a classical construction in graph theory, which we here naturally extend to weighted graphs}. 
We are not going to elaborate on this functorial viewpoint, but content ourselves with discussing in Section~\ref{sez5} a different, more sophisticated setting where  the same principle can be seen in action.

The case of combinatorial graphs is tightly related to the topic of random walk in random environments, see e.g.\ the classical surveys by Zeitouni~\cite{Zei02, Zei04}, which roughly speaking describe the behavior of a random walker who at each step finds herself moving in a new realization of a $d$-dimensional bond-percolation graph. 
 
At the same time, if the evolution of $\cL(t)$ is, in fact, deterministic, then~\eqref{e1} is essentially a non-autonomous evolution equation; well-posedness theory of such problems is a classical topic of operator theory, while some criteria for exponential stability have been recently obtained in~\cite{AreDieKra14} in the context of diffusion on metric graphs: in comparison with ours, the conditions therein are much more restrictive in that each realization of the considered graph is assumed to be connected.

The convergence of piecewise deterministic Markov processes (or random switching system) is discussed, in particular concerning the ergodicity of the Markov process \cite{Bakhtin2012, Cloez2015}. 
The results in \cite{Benaim2014} are concerned with the non-ergodicity of a switching system in the fast jump rate regime and open the path to similar results in \cite{Lawley2018}.

\medskip
{\bf Acknowledgment.} The authors would like to thank Jochen Glück  (Passau) and Marvin Plümer (Hagen) for their help in the proof of Lemma~\ref{lem:a1} and Walter Moretti (Trento) for several useful discussions concerning Lemma \ref{le:s.1}.

\section{Setting of the problem and main results}\label{sez:setting}

In this section we introduce a general setting for abstract random evolution problems: we will successively show that our motivating problem ({\bf P}) is but one special instance of a system that can be described in this way.

To begin with, we construct the random mechanism of switching by means of a semi-Markov process.
These processes have been introduced by Levy \cite{Levy1954} and Smith \cite{Smith1955} in order to overcome the limitation induced by the exponential distribution of the jump-time intervals and developed by Pyke \cite{Pyke1961, Pyke1964}. 
These models are widely used in the literature to model
random evolution problems and, more generally, evolution in random media, see \eg\ Korolyuk \cite{Korolyuk95}. 

Let $(Z(t),\ t \ge 0)$ be a semi-Markov process taking values in a set $E$, which denotes a given set of indices,
defined on a suitable probability space $(\Omega, \cF, \bP)$. 
By definition, this means that there exists a \textit{Markov renewal process} $\{(X_n, \tau_n): n\in\mathbb N\}$,
where $\{X_n\}$ is a jump Markov process with values in $E$ and $\{\tau_n\}$ are time intervals between jumps and,
if we introduce the counting process $N(t) := \max\{n\,:\, T_n \le t\}$, then
$Z(t) = X_{N(t)}$. 
The joint distribution is given by the transition probability function $q(x,y,t)$
\begin{align*}
q(x,y,t) = \bP(X_{n+1} = y, \tau_{n+1} < t \mid X_n = x). %Q -> q
\end{align*}
By definition, for fixed $t$, $(x,y) \mapsto q(x,y,t)$ is a sub-Markovian transition function, i.e., 
\begin{align*}
q(x,y,t)\ge 0\quad \hbox{and}\quad \sum_{z\in E} q(x,z,t) \le 1\qquad \hbox{for all }x,y\in E \hbox{ and all }t\ge 0.
\end{align*}

The non-negative random variables $\tau_n$ define the time intervals between jumps, while the Markov renewal times $\{ T_n,\ n \in \bN\}$ defined by
\begin{align*}
T_0 = 0,\qquad T_n = \sum_{k=1}^n \tau_k, \quad n\in \mathbb N
\end{align*} 
are the {\em regeneration times}.

For simplicity, in the sequel we assume that the components $X_n$ and $\tau_n$ are conditionally independent; therefore, the transition probability function can be represented in the form
\begin{align*}
q(x,y,t) = \pi(x,y) G_x(t),
\end{align*}
where $\big( \pi(x,y) \big)$ is a Markov transition matrix. 

Clearly, Markov chains and Markov processes with discrete state space are 
examples of semi-Markov processes (the first is associated with $\tau_n \equiv 1$, the second with independent, exponentially distributed $\tau_n$).  Our standing probabilistic assumptions are summarized in the following.

\begin{assumption}\label{hpZ}
$Z = (Z(t))_{t \ge 0}$ is a semi-Markov process based on a Markov renewal process $\{ (X_n,\tau_n):n\in\mathbb N\}$ over the state space $E\times [0,\infty)$ such that
\begin{enumerate}
\item the Markov transition matrix $\big( \pi(x,y) \big)$ defines an irreducible Markov chain with finite state space $E = \{1, \dots, N\}$;
\item the inter-arrival times $\tau_n$ are either constant, or the distribution functions $(G_x(t),\ t \ge 0)$, for every $x \in E$, has a finite continuous density function $g_x(t) > 0$ for a.e. $t > 0$; and
\item the inter-arrival times $\tau_n$ have finite expected value $\bE^x[\tau_n] = \mu_x > 0$.
\end{enumerate}
\end{assumption}

\begin{remark}\label{rem:itfllow}
Since the embedded Markov process $X$ is irreducible, there exists a unique invariant distribution $\rho = (\rho_1, \dots, \rho_N)$ for it.
\\
Moreover, this implies that 
the total time spent in any state by the semi-Markov process $Z$ is infinite almost surely, and the fraction of time spent in $x \in E = \{1, \dots, N\}$
satisfies
\begin{align}\label{e:tot-time}
\Theta_x \coloneqq \lim_{t \to \infty} \frac1t \int_0^t \uno_{\{Z(s) = x\}} \, {\rm d}s = \frac{\rho_x \mu_x}{\sum_{j \in E} \rho_j \mu_j}.
\end{align}
\end{remark}

Once our random environment has been described, we can  introduce the evolution problem.
\\
We consider an ensemble $\cK = \{ A_1, \dots, A_{N} \}$ of linear operators on a normed space $H$; clearly, the cardinality of $\cK$ is the same as that of the state space $E$ of the Markov chain.
\\
We can now introduce the abstract random Cauchy problem
\begin{equation}
\label{CPS}
\left\{
\begin{aligned}
\tfrac{d}{dt} u(t) &= A(Z(t)) u(t),\\
u(0) &= f,
\end{aligned}
\right.
\end{equation} 
where $A(Z(t)) = A_{X_n}$ for $t \in [T_n, T_{n+1})$. 
The solution of \eqref{CPS} is a random process, where the stochasticity enters the picture through the semi-Markov process $(Z(t))_{t\ge 0}$. 
Notice that \eqref{e1} is a special case of \eqref{CPS} on the finite-dimensional space $H = \bC^d$.

In the literature, (deterministic) non-autonomous Cauchy problems of the form \eqref{CPS} are are a classical topic with a well-developed theory, see e.g.~\cite{Tan79,AcqTer87,DauLio92}. 
In this paper, we shall use the following natural modification of the notion of solution.

\begin{definition}\label{def:soldeterm}
Assume that there exists a finite partition $0=T_0<T_1<\ldots<T_N=:T$ of $[0,T]$ such that $A(Z(t)) = A_{X_n}$ for all $t\in [T_{n-1}, T_n)$, $n=1,\ldots,N$.
We say that a càglàd function $u:[0,T]\to H$ is a {\em solution of \eqref{CPS} on $[0,T]$} if 
\begin{enumerate}
\item $u\in C^1((T_{n-1}, T_n);H)$ for all $n=1,\ldots,N$;
\item $u(t)\in D(A_{X_n})$ for all $t\in (T_{n-1}, T_n)$ and $n=1,\ldots,N$;
\item $u'(t) = A_{X_n} u(t) $ for all $t \in (T_{n-1}, T_n)$ and $n=1,\ldots,N$.
\end{enumerate}
\end{definition}

Sufficient conditions for well-posedness of~\eqref{CPS} are given by the following.

\begin{assumption}\label{hpA}
$H$ is a separable, complex Hilbert space and for every $j \in \{1, \dots, N\}$ the closed, densely defined operator $A_j : D(A_j) \subset H \to H$  generates a strongly-continuous, analytic semigroup of contractions and  it has no spectral values on $i\mathbb R$, with the possible exception of 0.
\end{assumption}

\begin{definition}\label{def:solstoch}
A {\em solution} $u$ for \eqref{CPS} is a stochastic process $\{u(t), t \ge 0\}$
\begin{align*}
\cF_t := \sigma\{ \{\tau_n \le t\} \cap \{(x_0, \dots, x_n) \in B\}, \; n \in \bN,\ B \in E^{n+1}\}.
\end{align*}
and whose trajectories solve the identity $u'(t) = A(Z(t))u(t)$ almost surely in the 
sense of Definition~\ref{def:soldeterm}.
\end{definition}

Existence and uniqueness of the solution in the sense of previous definition is a consequence of the well-posedness of the Cauchy problem driven by the operator $A_{X_n}$ on the time interval $(T_{n-1}, T_n)$.

\begin{theorem}\label{thm:wellp}
Under the Assumptions \ref{hpZ} and \ref{hpA}, given $f\in H$, \eqref{CPS} has a unique solution $u$, which can be expressed
as $u(t) = S(t)f$ in terms of the {\em random propagator}  $(S(t))_{t\ge 0}\subset {\mathcal L}(H)$ defined by
\begin{equation}
\label{e2}
S(t) := e^{(t-T_n)A_{X_n}} \prod_{k=0}^{n-1} e^{(T_{k+1} - T_k)A_{X_k}}, \qquad t \in [T_n, T_{n+1}),\ n\in\mathbb N.
\end{equation}
In particular, $u$ has continuous sample paths.
\end{theorem}

After establishing well-posedness of our abstract random Cauchy problem, we are interested in studying the long-time behavior of its solutions. To this purpose, we are going to impose the following.

\begin{assumption}\label{hpc}
$A_j$ has compact resolvent for every $j \in \{1, \dots, N\}$.
\end{assumption}

It follows from the Assumptions~\ref{hpA} and~\ref{hpc} that each $A_j$ has finite-dimensional null space, hence \textit{a fortiori} 
\[
K := \bigcap_{j=1}^N \ker A_j
\]
is finite-dimensional, too. If $k:=\dim K>0$, then we denote by $\{e_1,\ldots,e_k\}$ an orthonormal basis of $K$. 

We shall throughout denote by $P_K$ the orthogonal projector onto $K$ and $P_j$ the projector onto $\ker A_j$.
In general, for a projector $P$, its orthogonal operator is $P^\perp \coloneqq I - P$.
For the sake of consistency of notation, we use the same notation also in the case $K = \{0\}$.

\begin{remark}\label{rem:wlog}
In particular, it holds that  $A_j e_i=0$, for all $j=1,\ldots, N$ and all $i=1,\ldots,k$. 
Since the range of ${P_K}$ is spanned by null vectors of $A_j$ for each $j=1,\ldots,N$, ${P_K}$ commutes with each $A_j$, each semigroup operator $e^{t A_j}$, and each spectral projector ${P}_j$ onto $\ker A_j$, $j=1,\ldots,N$, $t\ge 0$.
\end{remark}

\begin{remark}\label{rem:8bis}
Let $A$ be an operator which satisfies our Assumptions \ref{hpA} and \ref{hpc}. 
Notice that they require $A$ to be dissipative and, thanks to Assumption \ref{hpc}, the spectrum of $A$ is discrete. 
By~\cite[Cor. IV.3.12 and Cor.~V.2.15]{EngNag00} there exists a spectral decomposition $H = H_0 \oplus H_d$ where $H_0 = \ker(A)$ and $H_d = H_0^\perp$ and the restriction of $A$ to $H_d$ generates an analytic contraction semigroup with strictly negative growth bound
$s_d(A) = \sup\{ \Re(\lambda) \,:\, \lambda \in \sigma(A) \setminus 0\} < 0$.
\end{remark}

In order to examine the long time behavior of the solution, we introduce a notion of convergence in the almost sure sense.

\begin{definition}\label{asconvergence}
We say that a random propagator $(S(t))_{t\ge 0}\subset {\mathcal L}(H)$ 
\emph{converges in norm $\bP$-almost surely towards a deterministic operator $M\in \mathcal L(H)$} if
\begin{align*}
\bP\left(\|\cdot\|-\lim_{t \to \infty} S(t) = M\right) = 1.
\end{align*}
\end{definition} 

Our main result can be expressed as follows.

\begin{theorem}\label{t1}
Under the Assumptions \ref{hpZ} , \ref{hpA}, and~\ref{hpc} the random propagator $(S(t))_{t\ge 0}$ for the Cauchy problem \eqref{CPS} converges in norm $\bP$-almost surely towards the orthogonal projector ${P_K}$ onto $K := \bigcap_{j=1}^N \ker A_j$.
\end{theorem}

Let us finally discuss the asymptotic behavior of the random evolution problem \eqref{CPS} under an additional assumption that was inspired by a result from~\cite{AreDieKra14}, where non-autonomous diffusion equations on a fixed network are studied.
Our aim is to study when the solution converges exponentially, for all initial data $f$, towards the orthogonal projector of $f$ onto the eigenspace with respect to the simple eigenvalue $0$. 
Adapting the ideas of~\cite{AreDieKra14} to our general setting, we shall impose the following.

\begin{assumption}\label{a-K}
{The null space of at least one operator in the ensemble $\mathcal K$, say $A_1$, agrees with $K= \bigcap_{j=1}^N \ker A_j$.}
\end{assumption}

It turns out that under this additional assumption $(S(t))_{t\ge 0}$ converges in norm \emph{exponentially} fast towards the orthogonal projector $P_K$.  We stress that this is a deterministic assertion, unlike that of Theorem \ref{t1}.

\begin{theorem}\label{th-adapt-ADFK}
Under the Assumptions \ref{hpZ}, \ref{hpA}, \ref{hpc}, and \ref{a-K} the random propagator $(S(t))_{t\ge 0}$ for the Cauchy problem \eqref{CPS} converges in norm $\bP$-almost surely towards the orthogonal projector ${P_K}$. The convergence is exponential with rate $\alpha$, where
	\begin{equation*}
	\alpha \ge \lim_{t \to + \infty} \frac{1}{t} \int_{0}^{t} \big(-s_d(A_1)\big) \uno_{(Z(s) = A_1)} \, {\rm d}s = \big(-s_d(A_1)\big)\Theta_1 > 0,
	\end{equation*}
where $s_d(A_1)$, introduced in Remark \ref{rem:8bis}, is strictly negative thanks to the Assumption \ref{a-K} and $\Theta_1$ where introduced in \eqref{e:tot-time}.
\end{theorem}

We postpone the proofs of our main results to Section~\ref{sez:proofs}. 

\begin{remark}
1) The Assumption~\ref{hpA} is especially satisfied if each $A_j$ is self-adjoint and negative semi-definite. 
In this case, moreover, $s_d(A_j) = \lambda_{k_j+1}(A_j)$ the largest non-zero eigenvalue ($k_j$ being the dimension of $\ker(A_j)$).

There are, however, further classes of operators satisfying it. If the semigroup generated by $A_j$ is positive and irreducible, for example, it follows from the Kre\u{\i}n--Rutman Theorem that the generator's spectral bound is a simple, isolated eigenvalue.

2)  We remark that the Assumption \ref{a-K} is not always satisfied: for instance, if all the null spaces $\ker A_j$ are one-dimensional but the intersection space $K$ is trivial. 
On the other hand, it is satisfied for all $\mathcal K$ that contain at least one operator $A_1$ with trivial null space.
\end{remark}

\subsection{Randomly switching heat equations}
\label{exa:neukredir}

The scope of our result is not restricted to graphs and networks.
To illustrate this, we consider a toy model -- a heat equation with initial data $f\in L^2(0,1)$, under different boundary conditions --  where the switching takes place at the level of operators, rather than underlying structures. 
Here we show that convergence to the projector onto the intersection of the null spaces holds. 
A more complex example, where the thermostat model with switching in the boundary conditions, is given in \cite{Lawley2018}: in that case, non-ergodicity is possible under certain conditions on the parameters.

\begin{enumerate}
\item 
We first consider two different realizations $A_1$, $A_2$ of the Laplacian acting on $L^2(0,1)$: with Neumann and with Krein--von Neumann boundary conditions, which lead to the domains
\begin{equation}
\label{domaincaseA}
\begin{split}
D(A_1):= \{u  \in H^2(0,1): u'(0) = u'(1) = 0\}
\end{split}
\end{equation}
and
\begin{equation}
\label{domaincaseC}
\begin{split}
D(A_2):=\left\{u \in H^2(0,1):u'(0) = u'(1) =u(1)-u(0)\right\}
\end{split}
\end{equation}
respectively, \cite[Exa.~14.14]{Sch12}. 
Both operators satisfy the Assumption~\ref{hpA}. 
Furthermore, the null space of the former realization is one-dimensional, as it consists of the constant functions; whereas a direct computation shows that null space of the latter realization is 2-dimensional, as it consists of all affine functions on $[0,1]$; hence the intersection $K$ of both null spaces is spanned by the constant function $\mathbbm 1$ on $(0,1)$. 
Both associated heat equations are well-posed, yet the latter is somewhat exotic in that the governing semigroup is not sub-markovian. 
We are interested in the long-time behavior of this mixed system~\eqref{CPS}, with $A(Z(t))\in \{A_1,A_2\}$: if the switching obeys the rule in the Assumption~\ref{hpZ}, the random propagator $(S(t))_{t\ge 0}$ converges in norm $\mathbb P$-almost surely towards the orthogonal projector onto the intersection of both null spaces, i.e., onto the space of constant functions on $[0,1]$; hence the solution of the abstract random Cauchy problem~\eqref{CPS} converges $\mathbb P$-almost surely towards the mean value of the initial data $f\in L^2(0,1)$.

\item 
On the other hand, if we aim at studying the switching between Dirichlet and Neumann boundary conditions, and thus introduce the realization $A_3$ with domain
\begin{equation}
\label{domaincaseD}
\begin{split}
D(A_3):= \{u  \in H^2(0,1): u(0) = u(1) = 0\},
\end{split}
\end{equation}
then one sees the intersection space $K$ is trivial, as $\ker A_3=\{0\}$, hence the random propagator converges in norm $\mathbb P$-almost surely to 0 if the Assumption~\ref{hpZ} is satisfied. 

\item Also observe that upon perturbing $A_3$ we find the new operator
\[
\begin{split}
\tilde{A}_3 u&:= A_3 u+\pi^2 u\\
D(\tilde{A}_3)&:=D(A_3),
\end{split}
\]
whose null space is now one-dimensional, as it is spanned by $\sin(\pi\cdot)$. 
Nevertheless, $\ker A_1\cap \ker \tilde{A}_3=\{0\}$, hence again under the Assumption~\ref{hpZ} the system switching between $A_1,\tilde{A}_3$ converges towards 0.

\item 
Finally, let us consider a switching between $A_1$ and $A_4$ defined as
\[
\begin{split}
A_4 u&:= \frac{d}{dx}\left( p\frac{du}{dx}\right)\\
D(A_4)&:=D(A_1),
\end{split}
\]
where $p\in W^{1,\infty}(0,1)$, $p(x)>0$ for all $x\in [0,1]$. Because $\ker A_1$ and $\ker A_4$ both agree with the space of constant functions, under the Assumption~\ref{hpZ} the random propagator converges in norm $\mathbb P$-almost surely towards the orthogonal projector onto the space of the constant functions.
\end{enumerate}

Moreover, as a consequence of Theorem \ref{th-adapt-ADFK} 
we can observe the exponential convergence of the random propagator for some (but not all) of these toy models. In particular, this holds whenever we take $A_3$ in the ensemble $\cK$: indeed, we have $K=\{0\}$ and then the Assumption \ref{a-K} is satisfied, since the first eigenvalue $\lambda_1^{(3)}$ of $A_3$ is strictly negative. 
The exponential convergence of $(S(t))_{t\ge 0}$ can be shown also for randomly switching systems where $\cK\subset \{A_1,A_2,A_4\}$.
In all of these cases $K$ agrees with the space of constant functions on $(0,1)$, hence it is one-dimensional and the Assumption \ref{a-K} is still fulfilled, since the second eigenvalue $\lambda_{2}^{(j)}<0,$ for $j=1,4$.
On the other hand, we cannot apply Theorem \ref{th-adapt-ADFK} and then prove the exponential convergence of the random propagator for all those models which switch $\tilde{A}_3$ with $A_1$ or/and $A_2$ or/and $A_4$. In fact, this implies that the intersection space $K$ is trivial again, but no one operator has strictly negative first eigenvalue.

\section{Technical lemmas and proofs}
\label{sez:proofs}
\subsection{A monotonicity lemma}
\label{sez4}

The following lemma \ref{lem:ineq-not-so}
provides the crucial tool to prove the assertion of Theorem \ref{t1}.
It shows how we can bound the norm of the random product of matrices which generates the random propagator $(S(t))_{t\ge 0}$ with respect to the stopping times.

Let $L\ge N$ and $(k_1, \dots, k_L)$ be a sequence of indices that covers the whole $E = \{1, \dots, N\}$. Given an ensemble $\cK$ of operators satisfying the Assumption~\ref{hpA}, let us consider the associated sequence of operators $(A_{k_1}, \dots, A_{k_L})$ taken from $\cK$.
We shall denote $P_j$ the projection on the kernel $\ker A_j$ and $P_K$ the projection on $K = \cap_{j=1}^L \ker A_{k_j} = \cap_{i=1}^N \ker A_i$.

\begin{remark}
In the proof we will need some known results in functional analysis: if $T$ is a compact operator on a reflexive Banach space $X$, then there exists $x$ belonging to the unit sphere of $X$ such that $\|T\|=\|Tx\|$, i.e., the norm of $T$ is attained: see \eg\ \cite[Corollary 1]{Acosta2006}. 
This is in particular true if $T=T(t)$ for some $t>0$, provided the semigroup generated by $A$ is analytic (or even merely norm continuous) and $A$ has compact resolvent, see~\cite[Thm.~II.4.29]{EngNag00}.
Moreover, the compact operators form a two-side ideal in $L(H)$.
\end{remark}

The following results are necessary steps in order to prove the main result of this section.

\begin{lemma}\label{lem:a1}
Let $(T(t))_{t\ge 0}$ be a contractive, analytic strongly continuous semigroup on a Hilbert space $H$ whose generator $A$ has compact resolvent and no eigenvalue on the imaginary axis, with the possible exception of $0$; let us denote by $P$ the orthogonal projector onto $\ker A$.

Then the following assertions hold:
\begin{enumerate}[(1)]
\item $\|T(t')x\|<\|T(t)x\|$ for all $x\not\in \ker A$ and all $t'>t\ge 0$;
\item $\ker A=\{x\in H:\|T(t_0)x\|=\|x\|\}$ for some $t_0>0$.
\end{enumerate}
\end{lemma}

\begin{proof}
(1) Fix $x\not\in \ker A$ and let $0 \le t < t'$.
	
Let us first consider the case of injective $A$, so that $P = 0$. Then $x \not= 0$, and $T(t)x \to 0$ as $t \to \infty$ by the Jacobs--deLeeuw--Glicksberg theory, see~\cite[Thm.~V.2.14 and Cor.~V.2.15]{EngNag00}. Due to analyticity  of the semigroup, the mapping $\varphi: (0,\infty) \ni t \mapsto \|{T(t)x}\|^2 \in \mathbb{R}$ is real analytic: indeed, for each $x\in H$ the mapping $(0,\infty) \ni t \mapsto T(t)x\in H$ is real analytic, hence it can be represented by an absolutely converging power series, say $T(t)x=\sum_{k=0}^\infty t^k f_k $; but then, the Cauchy product of $\sum_{k=0}^\infty t^k f_k $ with itself, given by $\sum_{m=0}^\infty t^m \sum_{l=0}^m (f_l,f_{m-l})$, is absolutely converging towards $\|{T(t)x}\|^2=(T(t)x,T(t)x)$.
	
If $\|T(t)x\| = \|T(t')x\|$, then $\varphi$ is constant on the interval $[t,t']$: 	indeed, by contractivity of the semigroup
\begin{align*}
\|T(s)x\|\le \|T(t)x\|=\|T(t')x\|\le \|T(s)x\|\qquad \hbox{for all }s\in [t,t'].
\end{align*}
Due to the identity theorem for real analytic functions, $\varphi$ is now constant on $(0,\infty)$ -- a contradiction, since $\varphi(t) \to \|x\|^2 \not= 0$ as $t \searrow 0$, but $\varphi(t) \to 0$ as $t \to \infty$. This proves the theorem in case that $P = 0$.
	
Let us now consider the case of general $P$: observe that
$Px\ne x$, since $x \not\in \ker A$.
Applying the first step of the proof to the restriction of $(T(t))_{t\ge 0}$ to the $H\ominus \ker A$, we see that
\begin{align*}
\|T(t)(I-P)x\|^2 > \|T(t')(I-P)x\|^2,
\end{align*}
hence by Pythagoras' theorem 
\begin{equation}\label{tfp}
\begin{split}
\|T(t)x\|^2& = \|T(t)Px\|^2 + \|T(t)(I-P)x\|^2 \\
&> \|T(t)Px\|^2 + \|T(t')(I-P)x\|^2\\
&= \|T(t')Px\|^2 + \|T(t')(I-P)x\|^2 = \|T(t')x\|^2.
\end{split}
\end{equation}
where the second to last identity holds because the fixed space of $(T(t))_{t\ge 0}$
\[
{\rm fix}(T(t))_{t\ge 0}:=\{x\in H:T(t)x=x\hbox{ for all }t\ge 0\}
\]
agrees with the null space of its generator $A$ by~\cite[Cor.~IV.3.8]{EngNag00}, hence $T(t)y=y$ for all $y\in \ker A$  and all $t\ge 0$.

(2) We see that
\[
\begin{split}
{\rm fix}(T(t))_{t\ge 0}&\subset \{x\in H:\|T(t)x\|=\|x\| \hbox{ for all }t\ge 0\}\\
&\subset \{x\in H:\|T(t_0)x\|=\|x\| \hbox{ for some }t_0\ge 0\}\\
&\stackrel{(1)}{\subset} \ker A.
\end{split}
\]
This concludes the proof, since as recalled before $\ker A={\rm fix}(T(t))_{t\ge 0}$.
\end{proof}

The following is probably linear algebraic folklore, but we choose to give a proof since could not find an appropriate reference.

\begin{lemma}\label{le:s.1}
Let $H$ be an Hilbert space {\color{black} and $P_1,\ldots,P_m$ be finitely many orthogonal projectors on $H$; let $P_K$ the orthogonal projector onto $\displaystyle K:=\bigcap_{i=1}^m \rg P_i$. If $P_i$ is compact for at least one $i=1,\ldots,L$, t}hen the operator $P_{k_L} \dots P_{k_1} P_K^\perp$ has norm strictly less than 1:
\begin{align*}
\|P_{k_L} \dots P_{k_1} P_K^\perp\| = 1-\varepsilon < 1.
\end{align*}
\end{lemma}

\begin{proof}
It is obvious that $\|P_{m} \dots P_{1} P_K^\perp\| \le 1$.
We proceed by contradiction and assume that 
\begin{equation}
\label{eq:s.contra}
\|P_{m} \dots P_{1} P_K^\perp\| = 1.
\end{equation}
Since at least one $P_i$ is  compact, so is the whole product, hence it is norm-attainable: there exists $x \in H$ with $\|x\| = 1$ such that $\|P_{m} \dots P_{1} P_K^\perp x\| = \|x\| = 1$.

Notice that
\begin{align*}
1 = \|P_{m} \dots P_{1} P_K^\perp x\| \le \|P_{m} \dots P_{1}\| \| P_K^\perp x\| 
\end{align*}
hence $\|P_K^\perp x\| = 1 = \|x\|$ and it follows that $x = P_K^\perp x$.
We then substitute in previous equality and get
\begin{align*}
1 = \|P_m \dots P_1 P_K^\perp x\| = \|P_m \dots P_1 x\| 
\end{align*}
and the same reasoning implies $\|P_1 x \| = 1$, and $x = P_1 x$.
Reiterating the same argument we obtain
$x = P_j x$ for any $j=1, \dots, m$, therefore $x \in K$; but we have $x = P_K^\perp x$, which implies $x=0$, a contradiction to $\|x\|=1$.
Therefore, \eqref{eq:s.contra} is false and the thesis follows.
\end{proof}

We now proceed to prove the main result of this section.
Recall that the operators $A_j$ are negative (non-positive) defined and $P_j$ is the projection on $\ker A_j$.

\begin{lemma}\label{lem:ineq-not-so}
In previous assumptions, for  $\eta > 0$ small enough there exists $\delta > 0$ such that, for  $t_i \ge \delta > 0$, $i=1, \dots, L$ we have
\begin{equation}
\label{bound}
||P_K^\perp e^{t_{L}A_{k_L} } \cdots e^{t_1 A_{k_1}}|| \le 1-\eta < 1.
\end{equation}
\end{lemma}

\begin{proof}
Recall that $P_i$ is the orthogonal projection on $\ker A_i$, $K = \cap_{i=1}^N \ker(A_i)$, and the projection on $K$ satisfies  $P_K P_i = P_K = P_i P_K$, $P_K^\perp P_i^\perp = P_i^\perp = P_i^\perp P_K^\perp$.
We have 
\[
\begin{split}
\| e^{t_L A_L} \dots e^{t_2A_2}e^{t_1A_1}P_K^\perp x \|^2
&= \| e^{t_L A_L} \dots e^{t_2A_2}e^{t_1A_1}(P_1 + P_1^\perp)P_K^\perp x \|^2
\\
&\le (1+\alpha) \| e^{t_L A_L} \dots e^{t_2A_2}e^{t_1A_1}P_1P_K^\perp x \|^2\\
&\qquad + (1 + \alpha^{-1}) \| e^{t_L A_L} \dots e^{t_2A_2}e^{t_1A_1}P_1^\perp P_K^\perp x \|^2 
\\
&\le (1+\alpha) \| e^{t_L A_L} \dots e^{t_2A_2} P_1P_K^\perp x \|^2 + (1 + \alpha^{-1})\| e^{t_L A_L} \dots e^{t_2A_2} \|^2 \|e^{t_1A_1}P_1^\perp x \|^2 
\end{split}
\]
where we use the fact that $e^{t_1A_1}P_1x = P_1x$ for any $x \in H$, $t_1 \ge 0${, and that $\ker A_1\supset K$, so $( \ker A_1)^\perp \subset K^\perp$}; {\color{black}the first estimate follows from Young's inequality}. Notice further that all semigroups involved are contraction operators, hence $\|e^{t_L A_L} \dots e^{t_2A_2}\|^2 \le 1$; 
finally, we have 
$\| e^{t_1A_1}P_1^\perp x \| \le e^{- t_1 \lambda_b(A_1)} \| P_1^\perp x \| \le e^{- t_1 \lambda_b(A_1)} \| x \|$.
Hence
\[
\| e^{t_L A_L} \dots e^{t_2A_2}e^{t_1A_1}P_K^\perp x \|^2
\le (1+\alpha) \| e^{t_L A_L} \dots e^{t_2A_2} P_K^\perp P_1 x \|^2 +  (1 + \alpha^{-1}) e^{- 2 t_1 \lambda_b(A_1)} \| x \|^2 .
\]
We continue by splitting the first term in the right hand side
\begin{align*}
&\| e^{t_L A_L} \dots e^{t_2A_2}e^{t_1A_1}P_K^\perp x \|^2
\\
&\le (1+\alpha) \| e^{t_L A_L} \dots e^{t_2A_2} (P_2 + P_2^\perp)P_K^\perp P_1 x \|^2 +  (1 + \alpha^{-1}) e^{- 2 t_1 \lambda_b(A_1)} \| x \|^2 
\\
&\le (1+\alpha)^2  \| e^{t_L A_L} \dots e^{t_2A_2} P_2 P_K^\perp P_1 x \|^2 + (1 + \alpha^{-1}) \| e^{t_L A_L} \dots e^{t_2A_2} P_2^\perp P_K^\perp P_1 x \|^2\\
&\qquad +  (1 + \alpha^{-1}) e^{- 2 t_1 \lambda_b(A_1)} \| x \|^2 
\\
&\le (1+\alpha)^2 \| e^{t_L A_L} \dots e^{t_3A_3} P_2 P_K^\perp P_1 x \|^2 + (1 + \alpha^{-1}) \| e^{t_L A_L} \dots e^{t_3A_3} \|^2 e^{- 2 t_2 \lambda_b(A_2)}  \| P_2^\perp P_K^\perp P_1 x \|^2\\
&\qquad +  (1 + \alpha^{-1}) e^{- 2 t_1 \lambda_b(A_1)} \| x \|^2
\\
&\le (1+\alpha)^2  \| e^{t_L A_L} \dots e^{t_3A_3} P_K^\perp P_2 P_1 x \|^2 + (1 + \alpha^{-1}) \left( e^{- 2 t_2 \lambda_b(A_2)} +  e^{- 2 t_1 \lambda_b(A_1)} \right) \| x \|^2 
\end{align*}
and by recursion, we finally obtain
\begin{align}
\label{eq:stima}
\| e^{t_L A_L} \dots e^{t_2A_2}e^{t_1A_1}P_K^\perp x \|^2
\le (1+\alpha)^L  \| P_L \dots P_3 P_2 P_1 P_K^\perp x \|^2 + (1 + \alpha^{-1})  \sum_{i=1}^L e^{- 2 t_i \lambda_b(A_i)}  \| x \|^2 
\end{align}
The operator in 
the first term is bounded in norm by $1-\varepsilon$, thanks to Lemma \ref{le:s.1}; therefore, we obtain the estimate
\begin{align*}
\| e^{t_L A_L} \dots e^{t_2A_2}e^{t_1A_1}P_K^\perp x \|^2
\le \left((1+\alpha)^L  (1 - \varepsilon)^2 + (1 + \alpha^{-1}) \sum_{i=1}^L e^{- 2 t_i \lambda_b(A_i)}  \right) \| x \|^2 
\end{align*}
The thesis follows by first taking $\alpha$ small enough such that the first addendum is bounded by $1-2\eta$, then taking $\delta$ large enough such that the second addendum is bounded by $\eta$.
\footnote{%
Let us notice that in formula \eqref{eq:stima}, the only fixed term is $\varepsilon$ from Lemma \ref{le:s.1}.
Thus, let us fix $\eta$ such that
\begin{align*}
\eta < \frac{1 - (1-\varepsilon)^2}{2} = \frac{\varepsilon(2-\varepsilon)}{2} %\qquad \eta=\varepsilon/2
\end{align*}
Then, we choose $\alpha$ such that
\begin{align*}
(1+\alpha)^L(1-\varepsilon)^2 = (1-2\eta)
\end{align*}
and, setting $\lambda_+ = \max \{\lambda_+^{(i)},\ :\ i = 1, \dots, L\} < 0$, we may choose
\begin{align*}
\delta > \frac{1}{|\lambda_+|} \log\left(\eta^{-1} L (1 + \alpha^{-1}) \right).
\end{align*}
}%
\end{proof}

Finally, we prove below that in case of a fixed, deterministic clock, the same result follows for arbitrary $\delta > 0$.

\begin{lemma}\label{lem:ineq-not-so-fixed-time}
Given an ensemble $\cK=\{A_1,\ldots,A_N\}$ of closed, densely defined, dissipative operators with compact resolvent that generate analytic strongly continuous semigroups on a Hilbert space $H$, let us consider the associated sequence of operators $(A_{k_1}, \dots, A_{k_L})$ taken from $\cK$.
If  $L\ge N$ and $(k_1, \dots, k_L)$ is a sequence of indices that covers the whole $E = \{1, \dots, N\}$, then for all $\delta>0$
\[
\|P_K^\perp e^{\delta A_{k_L}} \cdots e^{\delta A_{k_1}}\| < 1.
\]
\end{lemma}

\begin{proof}
Let us now prove the inequality by contradiction: because all semigroups as well as the projector $P_K^\perp$ are contractive and hence certainly $\|P_K^\perp e^{\delta A_{k_L}} \cdots e^{\delta A_{k_1}}\| \le 1$, it suffices to assume that 
\[
\|P_K^\perp e^{\delta A_{k_L}} \cdots e^{\delta A_{k_1}}\| =1
;
\] since the product operator is a compact operator, as stated before, there would then exist some $x\in H$, $x\ne 0$, with 
$\|P_K^\perp e^{\delta A_{k_L}} \cdots e^{\delta A_{k_1}}x\| =\|x\|$. Because 
\[
\|P_K^\perp e^{\delta A_{k_L}} \cdots e^{\delta A_{k_1}}x\| \le \|P_K^\perp e^{\delta A_{k_L}} \cdots e^{\delta A_{k_2}}\| \|e^{\delta A_{k_1}}x\| \le \|e^{\delta A_{k_1}}x\|,
\]
it follows that $\|e^{\delta A_{k_1}}x\|=\|x\|$ and hence, by Lemma~\ref{lem:a1}.(2), $x\in \ker A_{k_1}$, i.e., $e^{\delta A_{k_1}}x=x$. Proceeding recursively we see that $x\in \bigcap_{i=1}^L \ker A_{k_i}\subset K$, whence $e^{\delta A_{k_i}}x=x$ for all $i$ and hence
\[
\|x\|=\|P_K^\perp e^{\delta A_{k_L}} \cdots e^{\delta A_{k_1}}x\| =\|P_K^\perp x\|=0,
\]
a contradiction.
\end{proof}

\subsection{Proof of Theorem \ref{t1}}

We need two ingredients in this proof.
First, since the Markov chain $\{X_n\}$ is irreducible, there exists a cycle $\xi = (\xi_0, \xi_1, \dots, \xi_L = \xi_0)$ such that
\begin{itemize}
\item[-] the reached states cover the full set of indexes:
$$\{\xi_0,\ldots, \xi_{L-1}\}=\{1,\ldots,N\};$$
\item[-] the Markov chain follows this path with a strictly positive probability:
$$ p_{\xi_0, \xi_1} \cdots p_{\xi_{L-1}, \xi_{L}}  > 0;$$
\item[-] by a suitable rotation of the indexes, it is always possible to let $\xi_0 = X_0$. 
\end{itemize}

As the second main ingredient, we construct a new stochastic process $\{X'_m\}$ by considering the sequence $\{X_n\}$ divided in {\em blocks} of length $L$
\begin{align*}
X'_m := \left(X_{mL}, \ldots, X_{(m+1)L-1}\right), \qquad m \ge 0.
\end{align*}
By a standard argument, $\{X'_m\}$ is an irreducible Markov chain on the state space $E' = \{1,\ldots,N\}^L$.
Therefore, the state $\xi$ is recurrent and there exists an infinite subsequence $\{m_j\}$ of indices such that $X'_{m_j} = \xi$.
\\
Finally, we take a further subsequence $m_{j_l}$ such that {\em all the waiting times} $\{\tau_{m_{j_l}L}, \dots, \tau_{m_{j_l}(L+1)-1} \}$ are larger than the constant $\delta$ in Lemma \ref{lem:ineq-not-so}.
Thanks to Assumption \ref{hpZ},
also this subsequence diverges to infinity.
\\
Let $\tilde n(t) = \max\{l \ge 0 \,:\, T_{m_{j_l}(L+1)} \le t\}$ be the number of times the sequence $\{X'_n\}$ passes from the state $\xi$ with all the waiting times larger than $\delta$ up to time $t$.
By the above reasoning, $\tilde n(t) \to \infty$ as $t \to \infty$ (although this sequence may diverge very slowly).

Clearly, the number of times the chain $\{X_k\}$ follows the cycle $\xi$
is greater than (or equal to) the number of times $\{X_k'\}$ visits the state $\xi$ and, moreover,
the waiting times are larger than $\delta$ (as described above).
Thus,
\begin{align*}
\|P_K^\perp S(t) \| = \| P_K^\perp \prod_{k=0}^{N(t)} e^{\tau_{k+1} A_{X_k}} \|
\le \prod_{m=0}^{n(t)} \|P_K^\perp e^{\tau_{(m+1)L} A_{X_{(m+1)L-1}}} \dots e^{\tau_{mL+1} A_{X_{mL}}} \|
\end{align*}
where $n(t) = \lfloor N(t)/L \rfloor$ is the number of transitions of the Markov chain $\{X'_m\}$ that are completed up to time $t$; since all the operators on the right hand side have norm bounded by 1, we can further estimate
\begin{align*}
\|P_K^\perp S(t) \| 
\le \prod_{m=0}^{\tilde n(t)} \|P_K^\perp e^{\tau_{(m+1)L} A_{X_{(m+1)L-1}}} \dots e^{\tau_{mL+1} A_{X_{mL}}} \|
\end{align*}
From Lemma \ref{lem:ineq-not-so} we are able to estimate
\begin{align*}
\|P_K^\perp e^{\tau_{(m+1)L} A_{X_{(m+1)L-1}}} \dots e^{\tau_{mL+1} A_{X_{mL}}} \| \le 1-\eta <1
\end{align*}
hence
\begin{align}
\label{e:stethm}
\|{P_K^\perp} S(t)\| \le (1-\eta)^{\tilde n(t)} \xrightarrow{ t \to +\infty }0.
\end{align}

By Remark~\ref{rem:wlog} we can write the random propagator as
$$S(t)={P_K} \, S(t) + {P_K^\perp} \, S(t)={P_K}+ {P_K^\perp} \, S(t)$$
and the thesis
$$\lim_{t \rightarrow + \infty} \|S(t)-P_K\|=0 \qquad \mathbb{P}-\text{a.s.}$$
follows by \eqref{e:stethm}.
\hfill\qed	

\subsection{Proof of Theorem \ref{th-adapt-ADFK}}
As done in the proof of Theorem \ref{t1}, we can write the random propagator as
\begin{equation*}
S(t)=P_KS(t)+P_K^\perp S(t)
\end{equation*}
and show that $\|P_K^\perp S(t)\| \to 0$ as $t \to +\infty$ in order to obtain the thesis.
			
Denote $u(t) := S(t)f$ for all initial data $f$; we can estimate the norm of the vector $P_K^\perp u(t) \in H$ by
\begin{multline}
\label{eqproof-sez2}
||P_K^\perp u(t)||^2-||P_K^\perp f||^2 = \int_{0}^{t} \frac{d}{ds} ||P_K^\perp u(s)||^2 \, {\rm d}s
\\
= 2 \Re\left(\int_{0}^{t} ( \frac{d}{ds}P_K^\perp u(s), \ P_K^\perp u(s) ) \, {\rm d}s\right)
= 2\int_{0}^{t} \Re( A(Z(s))P_K^\perp u(s), \ P_K^\perp u(s) ) \, {\rm d}s,
\end{multline}
where the last equality holds due to Remark \ref{rem:wlog} and because $\frac{d}{ds}$ and $P_K^\perp$ commute.
We split the above integral with respect to the various states of $Z(t)$:
\begin{align*}
||P_K^\perp u(t)||^2-||P_K^\perp f||^2 = \sum_{j=1}^N 2\int_{0}^{t} \uno_{(Z(s) = j)} \, \Re ( A_{X_j} P_K^\perp u(s), \ P_K^\perp u(s) ) \, {\rm d}s,
\end{align*}
and since all the $A_k$'s are dissipative, we have the trivial estimate
\begin{align*}
||P_K^\perp u(t)||^2-||P_K^\perp f||^2 \le 2\int_{0}^{t} \uno_{(Z(s) = 1)} \, \Re ( A_{X_j} P_K^\perp u(s), \ P_K^\perp u(s) ) \, {\rm d}s.
\end{align*}
Now, by Remark \ref{rem:8bis}, the above becomes
\begin{align*}
||P_K^\perp u(t)||^2-||P_K^\perp f||^2 
&\le -2 s_d(A_1) \int_{0}^{t} \uno_{(Z(s) = 1)} \, ( P_K^\perp u(s), \ P_K^\perp u(s) ) \, {\rm d}s 
\\
&= -2 s_d(A_1) \int_{0}^{t} \uno_{(Z(s) = 1)} \, ||P_K^\perp u(s)||^2 \, {\rm d}s.
\end{align*}
By Gronwall's Lemma we deduce that
\begin{equation*}
\|P_K^\perp u(t)\|^2=\|P_k^\perp S(t)f \|^2 \le \|P_K^\perp f\|^2 \ e^{-2 s_d(A_1) \int_{0}^{t} \uno_{(Z( s) = 1)} \, {\rm d}s}.
\end{equation*}
The thesis now follows from Remark \ref{rem:itfllow}: indeed, the integral diverges to $+\infty$ $\bP$-almost surely, hence
$||P_K^\perp S(t)|| \to 0$.
\hfill\qed	

\section{Combinatorial graphs}
\label{sez3}

A simple (finite, undirected) combinatorial graph $\mG = (\mV, \mE)$ is a couple defined by a finite set $\mV$ of vertices $\mv$ and a subset $\mE \subset \mV^{(2)}$ of unordered pairs $\me:=\{\mv,\mw\}$ of elements of $\mV$; such a pair $\me$ is interpreted as the edge connecting the vertices $\mv,\mw$.

Given a simple graph $\mG = (\mV, \mE)$, let us introduce a positive weight function on the set of vertices $\mV$
\begin{align*}
m \colon \mV \to (0,+\infty)
\end{align*}
which induces the scalar product 
\begin{align*}
(f,g)_m:= \sum_{\mv \in \mV} m(\mv)f(\mv)\overline{g(\mv)},\qquad f,g\in \bC^\mV
\end{align*}
on the space $\bC^\mV$ of complex valued functions $f \colon \mV \to \bC$: we denote by $\ell^2_m(\mV)$ the Hilbert space $\bC^\mV$ with respect to $(\cdot,\cdot)_m$.
In addition, let 
\begin{align*}
\mu \colon \mE \to (0,+\infty)
\end{align*}
be a positive weight function on the set of edges $\mE$. We call the 4-tuple $(\mV, \mE, m, \mu)$ a {\em weighted combinatorial graph}. 

\begin{remark}\label{rem:connect}
We stress that each weighted graph is a metric space with respect to the shortest path metric; while the topology does depend on the weights, any two weights define equivalent topologies, and in particular it does not depend on $m,\mu$ whether $\mG$ is connected or not.
\end{remark}

Let us recall the notion of \textit{discrete Laplacian} (or \textit{Laplace--Beltrami matrix}) \textit{$\cL_{m,\mu}$ on a weighted graph $\mG=(\mV,\mE,m,\mu)$}, cf.~\cite[\S~2.1.4]{Mug14} -- or shortly: \textit{weighted Laplacian}. For any vertex $\mv\in\mV$, let $\mE_\mv$ denote the set of all edges having $\mv$ as an endpoint. Then $\cL_{m,\mu}  \colon \ell^2_m(\mV) \to \ell^2_m(\mV)$ {is defined by
\begin{align*}
\cL_{m,\mu} f(\mv) := \frac{1}{m(\mv)}\sum_{\me=\{\mv,\mw\} \in \mE_\mv} \mu(\me) \left( f(\mw)-f(\mv) \right), \qquad \forall\, \mv \in \mV;
\end{align*}
$\cL_{m,\mu}$ reduces to the discrete, negative semi-definite Laplacian if $\mu\equiv1$ and $m\equiv 1$; i.e., $\cL_{1,1}$ is \textit{minus} the Laplacian matrix that is common in the literature~\cite[\S~2.1.4]{Mug14}. Indeed, we stress that we have \textit{not} adopted the usual sign convention of algebraic graph theory, as any such $\mathcal L_{m,\mu}$ is self-adjoint and \textit{negative} semi-definite. More generally, $\mathcal L_{m,\mu}$ satisfies the Assumptions~\ref{hpA} and~\ref{hpc} and it can be shown that $\cL_{m,\mu}$ (and not $-\cL_{m,\mu}$) generates a Markovian semigroup.} 
The associated sesquilinear form $q \colon \ell^2_m(\mV) \times \ell^2_m(\mV) \to \bC $ 
	$$q(f,g)=\sum_{\me=\{\mv,\mw\} \in \mE} \mu(\me) \left(f(\mv)-f(\mw)\right) \overline{\left(g(\mv)-g(\mw)\right)}, \qquad f,g \in \ell^2_m(\mV),	$$
such that
	$$q(f,g)=(\cL_{m,\mu} f,g)_m=(f,\cL_{m,\mu} g)_m, \qquad f,g \in \ell^2_m(\mV);$$
accordingly, its Rayleigh quotient is
\begin{equation}
\label{RQ}
	\frac{(f,\cL_{m,\mu} f)_m}{(f,f)_m} = \frac{q(f,f)}{\|f\|^2_m}=\frac{\sum_{\me=\{\mv,\mw\} \in \mE} \mu(\me) |f(\mv)-f(\mw)|^2}{\|f\|^2_m},\qquad f\ne 0.
\end{equation}
It follows from \eqref{RQ} that $\lambda = 0$ is an eigenvalue of each weighted Laplacian $\cL_{m,\mu}$: the associated eigenfunctions are constant on each connected component of $\mG=(\mV,\mE,m,\mu)$. 
Therefore, it turns out that the null space of $\cL_{m,\mu}$ agrees with the null space of the unweighted Laplacian (on $(\mV,\mE)$) associated with $\mG$.

\subsection{The general model}\label{sec:generalcomb}

Throughout this section we consider a finite collection $\cC$ of graphs. 

\begin{assumption}\label{hp3.1}
 $\cC = \{\mG_1, . . . , \mG_N \}$, where $\mG_1=(\mV,\mE_1,m_1,\mu_1), . . . , \mG_N=(\mV,\mE_N,m_N,\mu_N)$ are simple graphs with same vertex set $\mV$ but possibly different edge sets $\mE_i$, vertex weights $\mu_i$, and edge weights $\mu_i$, $i=1,\ldots,N$.
\end{assumption}
 
The following seems to be natural but not quite standard: we prefer to note it explicitly.

\begin{definition}[Union and intersection of weighted graphs]
\label{de:ungr}
The \emph{union} of $\mG_i=(\mV,\mE_i,m_i,\mu_i)$, $i=1,\ldots,N$,  is the weighted graph $\mG_\cup=(\mV,\mE,m,\mu)$ with set of vertices $\mV$, set of edges $\mE := \bigcup_{i=1}^N \mE_i$, vertex weights
\begin{align*}
m(\mv) := \min_{i=1,\ldots,N} m_i(\mv), \qquad \mv \in \mV,
\end{align*}
and edge weights
\begin{align*}
\mu(\me) := \max_{i=1,\ldots,N} \mu_i(\me), \qquad \me \in \mE.
\end{align*}
Likewise, the \emph{intersection} of $\mG_i=(\mV,\mE_i,m_i,\mu_i)$, $i=1,\ldots,N$,  is the weighted graph
$\mG_\cap=(\mV,\mE,m,\mu)$ with set of vertices $\mV$, set of edges $\mE := \bigcap_{i=1}^N \mE_i$, vertex weights
\begin{align*}
m(\mv) := \max_{i=1,\ldots,N} m_i(\mv), \qquad \mv \in \mV,
\end{align*}
and edge weights
\begin{align*}
\mu(\me) := \min_{i=1,\ldots,N} \mu_i(\me), \qquad \me \in \mE;
\end{align*}
here we set $\mu_i(\me):= 0$ if $\me \not\in \mE_i$.
\end{definition}

In this way, it is possible to study the behavior of the intersection of the null spaces
of the Laplacian operators $\cL_{m_k,\mu_k}(\mG_k)$ associated with the graphs in $\cC$.
This result seems interesting on its own, since it explicitly connects the geometry of the graph with the algebraic property of the Laplacian operator.

\begin{lemma}
\label{lemmanull space}
Given $\mG_1,\ldots,\mG_N$ combinatorial graphs satisfying the Assumption~\ref{hp3.1}, let $\mG$ be their weighted union graph (see Definition~\ref{de:ungr})  and let $\cL_{m,\mu}(\mG)$ be the discrete Laplacian on $\mG$.
Then
\begin{align*}
\ker \cL(\mG) = 
\bigcap_{i=1}^{N} \ker \cL_{m_i,\mu_i}(\mG_i).
\end{align*}
\end{lemma}

\begin{proof}
It suffices to work with unweighted graphs; 
for simplicity, moreover, we only prove the case $N=2$.
In general, suppose that $\mG$, $\mG_1$ and $\mG_2$ are written as disjoint unions of connected components:
\begin{align*}
\mG=\bigsqcup_{h=1}^{l} \mG^{(h)},\hspace{0,7cm}\mG_1= \bigsqcup_{j=1}^{m}\mG_1^{(j)},\hspace{0,7cm} \mG_2=\bigsqcup_{k=1}^{n} \mG_{2}^{(k)}.
\end{align*}
Each connected component of the union $G$ can be expressed as
\begin{equation*}
\mG^{(h)}=\bigsqcup_{j \in C_1^{(h)}} \mG_1^{(j)}\hspace{0,4cm}\text{or}\hspace{0,4cm} \mG^{(h)}=\bigsqcup_{k \in C_2^{(h)}} \mG_2^{(k)},\qquad h=1,\ldots,l,
\end{equation*}
where
\begin{align*}
&C_1^{(h)}:= \left\lbrace j \in \{1,\ldots,m\}\hspace{0,1cm}:\hspace{0,1cm}\mG_1^{(j)} \subseteq \mG^{(h)}\right\rbrace,\\
&C_2^{(h)}:= \left\lbrace k \in \{1,\ldots,n\}\hspace{0,1cm}:\hspace{0,1cm} \mG_2^{(k)} \subseteq \mG^{(h)}\right\rbrace.
\end{align*}
		
Any eigenfunction associated with the null eigenvalue of $\cL$ shall be constant on any connected component of $\mG$, hence $\mathcal{B}_{\mG}=\{ \mathbbm{1}_{h}, \ h=1,\ldots,l \}$, where
\begin{equation*}
\forall \, \mv \in \mV: \qquad \mathbbm{1}_h(\mv) := \begin{cases}
1 & \text{if } \mv \in \mG^{(h)},\\
0 & \text{otherwise},
\end{cases}
\end{equation*}
is a basis for $\ker \mathcal{L}(\mG)$. 
Similarly, $\mathcal{B}_1=\{\mathbbm{1}_{1,j}, \ j=1,\ldots,m\}$ and $\mathcal{B}_2=\{\mathbbm{1}_{2,k}, \ k=1,\ldots,n\}$ are a basis of $\ker \mathcal{L}(\mG_1)$ and $\ker \mathcal{L}(\mG_2)$, respectively. 
Hence, we only need to prove that for all $h=1,\ldots,l$, $\mathbbm{1}_h$ is in the intersection of the null spaces in $\mathcal{C}$ and then extend the result to $\ker \mathcal{L}(\mG)$ by linearity. 
In particular, by construction the function $ \mathbbm{1}_h$ will be
\begin{equation*}
\forall \, \mv \in \mV: \qquad \mathbbm{1}_h(\mv) := \begin{cases}
1 & \text{if }\mv \in \mG_1^{(j)}, \ j \in C_1^{(h)},\\
0 & \text{otherwise},
\end{cases}
\end{equation*}
and 
\begin{equation*}
\forall \, \mv \in \mV: \qquad \mathbbm{1}_h (\mv):= \begin{cases}
1 & \text{if }\mv \in \mG_2^{(k)}, \ k \in C_2^{(h)},\\
0 & \text{otherwise,}
\end{cases}
\end{equation*}
thus $\mathbbm{1}_h \in \ker \mathcal{L}(\mG_1) \cap \ker \mathcal{L}(\mG_2),$ in fact it can be written as linear combination of both basis $\mathcal{B}_1$ and $\mathcal{B}_2$ as 
\begin{equation*}
\mathbbm{1}_h = \sum_{j \in C_1^{(h)}} \mathbbm{1}_{1,j} \hspace{0,4cm}\text{or}\hspace{0,4cm}\mathbbm{1}_h=\sum_{k \in C_2^{(h)}} \mathbbm{1}_{2,k}.
\end{equation*}
		
On the other hand, given $f \in \ker \mathcal{L}(\mG_1) \cap \ker \mathcal{L}(\mG_2)$, we have 
\begin{equation}\label{eq1}
f = \alpha_1 \mathbbm{1}_{1,1}+ \cdots + \alpha_m \mathbbm{1}_{1,m}
\end{equation}
and
\begin{equation}
\label{eqqq2}
f = \beta_1 \mathbbm{1}_{2,1}+ \cdots + \beta_n \mathbbm{1}_{2,n},
\end{equation}
where $\mathcal{B}_1$ and $\mathcal{B}_2$ as above. 
Then, comparing the expressions \eqref{eq1} and \eqref{eqqq2}, we get that
\begin{equation*}
\begin{split}
& \alpha_j = \beta_k = c_1, \qquad \forall j \in C_1^{(1)}, \ \forall k \in C_2^{(1)},\\
& \phantom{\alpha = a}\vdots \\
& \alpha_j = \beta_k = c_l, \qquad \forall j \in C_1^{(l)}, \ \forall k \in C_2^{(l)}
\end{split}
\end{equation*}
and $f$ can also be expressed in terms of $\mathcal{B}_\mG$ as
$$f=\sum_{h=1}^{l} c_h \mathbbm{1}_h,$$
thus $f \in \ker \mathcal{L}(\mG).$
\end{proof}

As a corollary, we notice the following result, concerning the relation between the connectedness of the union graph $\mG$ (see Remark~\ref{rem:connect}) and the dimension of the kernel of the (weighted or unweighted) Laplacian operator.

\begin{corollary} \label{Lemmaunione} 
Given $\mG_1,\ldots,\mG_N$ combinatorial graphs satisfying the Assumption~\ref{hp3.1}, let $\mG$ be their weighted union graph. Then
\begin{equation}
\label{e:1Lemmaunione}
	\mG \text{ is connected } \ \Longleftrightarrow \ \bigcap_{i=1}^{N} \ker \cL_{m_i,\mu_i}(\mG_i) = \ \langle \mathbbm{1} \rangle.
\end{equation}
\end{corollary}

Partially motivated by Corollary~\ref{Lemmaunione}, with a slight abuse of notation we adopt in the following the notation $\mathcal L_i:=\mathcal L_{m_i,\mu_i}(\mG_i)$, $i=1,\ldots,N$.
Fixed a probability space $(\Omega,\cF,\bP)$, let $(Z(t))_{t \ge 0}$ be a semi-Markov process on the state space $E = \{1,\dots, N\}$ which satisfies Assumption \ref{hpZ}. 
In this section, we shall consider the random Cauchy problem
\begin{equation}\label{REE}
\left\{
\begin{aligned}
\tfrac{du}{dt}(t, \mv) &= \cL_{X_k} u(t,\mv),\qquad &&\mv \in \mV,\ t \in [T_k, T_{k+1}), \ k\in\mathbb N,\\
u(0,x) &= f(\mv), &&\mv \in \mV,
\end{aligned}
\right.
\end{equation}
where $\cL_{X_k}$ is the discrete Laplace operator associated with the currently selected graph $\mG_{X_k}$.
The above equation is also known in the literature as the (random) discrete heat equation.
Here and in the sequel, we neglect the dependence on $\omega$ to simplify the notation.

We can now state our main result in this section. 

\begin{theorem}\label{TEOR1}
Let $(Z(t))_{t\ge 0}$ be a semi-Markov process and $\cC$ be a family of graphs that satisfy the Assumptions \ref{hpZ} and \ref{hp3.1}, respectively.
Then the random propagator $(S(t))_{t\ge 0}$ for the Cauchy problem~\eqref{REE} converges in norm $\mathbb P$-almost surely towards the orthogonal projector ${P_K}$ onto the space
$K = \displaystyle \bigcap_{i=1}^N \ker \cL_i$.
\end{theorem}

We can now see that the connectedness of the union graph is a necessary and sufficient condition for the convergence of the random propagator $(S(t))_{t\ge 0}$ towards the orthogonal projector $P_0$ onto the space of constant functions on $\mV$, \ie, the eigenspace $\langle \mathbbm{1}\rangle$ spanned by $\mathbbm{1}$. 
Notice that since $\mathbbm{1}$ is an eigenvector for each $\mathcal L_k$, $P_0$ commutes with each $\cL_k$ and each $e^{t \cL_k}$, $k=1,\ldots,N$, $t\ge 0$.

\begin{corollary}
\label{teoremaConnesso}
Under the assumptions of Theorem~\ref{TEOR1}, $(S(t))_{t\ge 0}$ converges in norm $\bP$-almost surely to $P_0$ if and only if the union graph $\mG$ is connected.
\end{corollary}

\begin{proof}
Corollary \ref{Lemmaunione} implies that $P_0 = {P_K}$ if and only if the union graph $\mG$ is connected, 
while Theorem \ref{TEOR1} implies the convergence of $S(t)$ towards ${P_K}$, hence the sufficiency and necessity of the condition.
\end{proof}

In the last part of this section we present two special cases of evolution on combinatorial graphs where we discuss the relation between our result and the existing literature.

\subsection{Connected graphs}

In this section we assume that all the graphs in $\cC$ are connected.
As we have already seen, this assumption is unnecessarily strong if we aim at solving ({\bf P}).
\\
However, we are going to show an interesting link between our problem and the analysis of LCP-sets.
For simplicity, in this section we assume that $\tau_n = 1$ for every $n$, hence $T_n = n$ and $Z(t) = Z(\lfloor t \rfloor) = X_{\lfloor t \rfloor}$.

A set $\cK$ of matrices is said to have the \textit{left-convergent product property}, or simply to be an \textit{LCP set}, if 
for every sequence $\mj = (j_n)_{ n \in \bN}$ the infinite left-product $M_\mj := \displaystyle \prod_{k=0}^\infty M_{j_k}$ converges.
It is known~\cite{ElsKolNeu90} that $\cK$ is an LCP set if $\cK$ is \textit{paracontracting}, meaning that for some matrix norm
\[
Mx\ne x \quad \Rightarrow \quad \|Mx\|<\|x\|\qquad\hbox{ for all $M\in \cK$ and $x\in \bR^d$}	.
\]
The issue of convergence of infinite products of matrices has been finally settled in a fundamental paper by Daubechies and Lagarias: in particular, see~\cite[Thm.~4.1 and Thm.~4.2]{DauLag92} and also the erratum in~\cite{DauLag01}.

\begin{proposition}\cite[Thm.~4.2]{DauLag92}\label{dala2}
Let $\cK$ be a finite set of $d \times d$ matrices. Then the following are equivalent.
\begin{enumerate} %[(a)]
\item[(a)] $\cK$ is an LCP set whose limit function $\mj \mapsto M_\mj$ is continuous with respect to the sequence topology on $\bS = \{\mj = (j_n)_{n \in \bN}\}$.
\item[(b)] All matrices $M_i$ in $\cK$ have the same eigenspace $E_1$ with respect to the eigenvalue $1$,  this eigenspace is simple for all $M_i$,
 and there exists a vector space $V$ such that $\mathbb C^{d} = E_1 \oplus V$ and such that if $P_V$ is the oblique projector onto $V$ away from $\me_1$, then $P_V \cK P_V$ is an LCP set whose limit function is identically 0.
\end{enumerate}

In particular, if $E_1$ is a 1-dimensional subspace, then the limit function $M$ is the projector onto this space.
\end{proposition}

Now we can state this result in the setting of combinatorial graphs. 
Under the assumption of connectedness of all graphs, Theorem \ref{t1} states that  $S(t)$ will converge to $P_0$ (the projector on the subspace $\langle \uno \rangle$ of constant functions) no matter which sequence of graphs we follow in \eqref{REE},
thus it provides the same result as in the deterministic case treated in Proposition \ref{dala2}.
We shall give in Lemma~\ref{connectedgraphsprop} an alternative proof to this result, which specializes to the notation of graph theory.

\begin{lemma}
\label{connectedgraphsprop}
Let $\mathcal{C}=\{\mG_1, \ldots, \mG_N\}$ be a finite family of connected graphs and $Z = (X_n, \tau_n=1)$ be an irreducible Markov chain. Then
\begin{equation}
\label{convergenceforallomega}
\lim_{t \rightarrow + \infty} \|S(t)-{P}_0\|=0\qquad \hbox{along any trajectory }\omega \in \Omega
\end{equation}
holds for the random evolution problem \eqref{REE}.
\end{lemma}

\begin{proof}
In our assumptions, $0$ is a simple eigenvalue of each Laplacian matrix $\mathcal L_k:=\mathcal L(\mG_k)$ with associated eigenvector $\mathbbm{1}$.
The orthogonal operator $P_0^\perp$ is again an orthogonal projector operator
with range $\langle \mathbbm{1} \rangle^{\perp}$.
Notice that $P_0 S(t) = P_0$ because $\rg P_0 = \langle \uno \rangle$ is contained in 
\[
\hbox{ fix }(e^{t \cL_k})_{t\ge 0}:=\{x:\mV\to \mathbb C: e^{t \cL_k}x=x\hbox{ for all }t\ge 0\}
\]
 for every $1 \le k \le N$.
Therefore,
\begin{equation}
\label{primopezzo}
S(t)=P_0S(t) + (I-P_0) S(t) = P_0 + P_0^\perp S(t)\qquad \hbox{for all }t\ge 0, 
\end{equation}
and
we can prove the assertion by showing that
\begin{equation*}
\lim_{t \rightarrow + \infty} \| P_0^\perp S(t) \|= 0.
\end{equation*}
First of all, by definition $P_0^\perp$ is idempotent and commutes with the exponential matrix of every Laplace operator. Hence
\begin{equation*}
P_0^\perp S(t) =P_0^\perp e^{(t-k)\mathcal{L}_{X_k}}e^{\mathcal{L}_{X_{k-1}}} \cdots e^{\mathcal{L}_{X_0}}
 = P_0^\perp e^{(t-k)\mathcal{L}_{X_k}} P_0^\perp e^{\mathcal{L}_{X_{k-1}}} \cdots P_0^\perp e^{\mathcal{L}_{X_0}}.
\end{equation*}
We claim that
\begin{equation}
\label{e.Lemmafuturo}
\text{ each matrix $P_0^\perp e^{\mathcal{L}_i}$, $i=1,\ldots,N$ has norm strictly less than 1.}
\end{equation}
By the finiteness of $\mE$, we denote by 
$$\delta := \max \left\lbrace ||P_0^\perp e^{\mathcal{L}_i}|| \hspace{0,2cm}:\hspace{0,2cm}i=1,\ldots,N \right\rbrace < 1.$$
For all $t>0,$ let $k \in \mathbb{N}$ be such that $k \le t < k+1$. By sub-multiplicativity of the matrix norm we have
\begin{equation*}
\begin{split}
||P_0^\perp S(t)||&=|| P_0^\perp e^{(t-k)\cL_{X_k}} P_0^\perp e^{\mathcal{L}_{X_{k-1}}}\cdots P_0^\perp e^{\mathcal{L}_{X_0}}||\\
&\le ||P_0^\perp e^{(t-k)\mathcal{L}_{X_k}}||\hspace{0,1cm} ||P_0^\perp e^{\mathcal{L}_{X_{k-1}}}|| \cdots ||P_0^\perp e^{\mathcal{L}_{X_0}}||\\
&\le ||P_0^\perp e^{\mathcal{L}_{X_{k-1}}}|| \cdots ||P_0^\perp e^{\mathcal{L}_{X_0}}|| \le \delta^{k-1}.
\end{split}
\end{equation*}
If $t \to +\infty$, then $k \to +\infty$ and we finally get
$$\lim_{t \rightarrow + \infty} \| P_0^\perp S(t)\|=0$$
which implies the thesis.

In order to complete  the proof it remains to show that claim \eqref{e.Lemmafuturo} holds.
We have proved a more refined version of this claim in Lemma \ref{lem:ineq-not-so}; however, in the current setting, the proof is straightforward.
Let $\cL$ denote the Laplacian operator for a connected graph $\mG$.
By a direct computation we have for all $t>0$
\begin{align*}
\|(I-{P_0}) e^{t \cL} f\|^2 = \sum_{k=2}^{d} e^{2 t \lambda_k} ( f, e_k )^2_{\ell^2}
 \le e^{2 t \lambda_{2}} \sum_{k=2}^{d} ( f, e_k )^2_{\ell^2}
= e^{2 t \lambda_{2}} \|(I-{P_0}) f\|^2
 \le e^{2 t \lambda_{2}} \|f\|^2,
\end{align*}
whence $\|(I-{P_0}) e^{\delta \cL} \|^2<e^{2 \delta \lambda_{2}} <1$
since $\lambda_{2} < 0$.
\end{proof} 

\subsection{Randomly switching {\color{black}combinatorial graphs} with non-zero second eigenvalue} \label{sezKdiscrgraph}

The goal here is to 
apply our exponential convergence criteria to combinatorial graphs. Consider the random evolution problem \eqref{REE}; we are going to show the exponential convergence of the random propagator $(S(t))_{t\ge 0}$ provided that the following assumption holds:

\begin{assumption}\label{a-nat}
At least one of the graphs in $\cC$, say $\mG_1$, is a connected graph.
\end{assumption}

It follows that the union graph $\mG$ is connected, too, hence the intersection space $K$ is one-dimensional and $P_K = P_0$ is the projection onto the space of constant functions on $\mV$. 
Moreover, this means that $\lambda_2(\cL_1) < 0$. Therefore, the Assumption \ref{a-K} is satisfied and we can directly apply Theorem \ref{th-adapt-ADFK}.

\begin{corollary}\label{cor-K-combgraphs}
Let $(Z(t))_{t\ge 0}$ be a semi-Markov process and $\cC$ be a family of graphs that satisfy the Assumptions \ref{hpZ} and \ref{hp3.1}, respectively. Let additionally the Assumption~\ref{a-nat} hold. 

Then the random propagator $(S(t))_{t \ge 0}$ converges in norm $\bP$-almost surely exponentially fast towards the orthogonal projector $P_0$ with an exponential rate no lower than
\begin{align*}
\alpha = -\sum_{j=1}^N \lambda_2(\cL_j) \Theta_j
\end{align*} 
that is the average of the eigenvalues  $\lambda_{2}(\cL_j)$ with respect to the fraction of time $\Theta_j$ spent by the process $Z(t)$ in the various states.
\footnote{it is possible to explicitly compute $\Theta_1$ in terms of the invariant distribution $\rho = (\rho_1, \dots, \rho_N)$ associated to the embedded Markov chain $X$ and the expected values of the jump times for the different states $\mu_j = \bE^{j}[\tau_1]$ by the formula
\begin{align*}
\displaystyle \Theta_j = \frac{\rho_j \mu_j}{\sum_{l=1}^N \rho_l \mu_l}.
\end{align*}
}
\end{corollary}

\begin{proof}
The assertion follows from Theorem \ref{th-adapt-ADFK}. Notice that the exponential rate can be computed by
\begin{align*}
-\frac{1}{t} \int_0^t \lambda_2(\cL_{Z(s)}) \, {\rm d}s = -\sum_{j=1}^N \frac{1}{t} \int_0^t \uno_{(Z(s) = \mG_j)}  \, {\rm d}s \, \lambda_2(\cL_{j}) 
\end{align*}
which converges, as $t \to \infty$, to (compare Remark \ref{rem:itfllow})
\begin{align*}
\alpha = -\sum_{j=1}^N \Theta_j \lambda_2(\cL_{j}) = -\sum_{j=1}^N \lambda_2(\cL_{j}) \frac{\rho_j \mu_j}{\sum_{l=1}^N \rho_l \mu_l}.
\end{align*}
This concludes the proof.
\end{proof}

\begin{remark}\label{rem:unioninters}
Adapting the proof of~\cite[Cor.~3.2]{Fie73} (where the convention is adopted that $\mathcal L$ is \textit{positive} semi-definite)  we see that 
each of the discrete Laplacians $\cL_k$ on the weighted combinatorial graph $\mG_k$ has second largest eigenvalue $\lambda_2(\cL_k) := \lambda_2(\mG_k)\in [\lambda_2(\mG_\cup),\lambda_2(\mG_\cap)]$, where $\mG_\cup,\mG_\cap$ are the union and intersection graph introduced in Lemma \ref{lemmanull space}, respectively:
therefore we conclude that the convergence to equilibrium for the randomly switching problem is not faster (resp., not slower) than in the case of the heat equation on  $\mG_\cup$ (resp., on $\mG_\cap$; observe that $\mG_\cap$ may however be disconnected, and hence $\lambda_2(\mG_\cap)$ may vanish, even if all $\mG_k$ are connected).

Estimates on the rate of convergence to equilibrium of the random propagator are readily available: it is well-known that, for a generic unweighted connected graph $\mG$, $-|\mV|\le \lambda_2(\mG)\le -2(1-\cos \frac{\pi}{|\mV|})$, where the second inequality is an equality if and only if $\mG$ is a path graph, see~\cite[3.10 and 4.3]{Fie73}.
It follows that $\lambda_2(\mathcal L_k)\in [-|\mV|,-2(1-\cos \frac{\pi}{|\mV|})]$ if in particular $\mG_\cap$ is connected; this gives an estimate on the convergence rate in Corollary~\ref{cor-K-combgraphs}.
\end{remark}

\section{Metric graphs}
\label{sez5}
In this section we discuss the application of Theorem \ref{t1} to  finite {\em metric graphs}. 
Roughly speaking, metric graphs are usual graphs (as known from discrete mathematics) whose edges are identified with real intervals -- in this case, finitely many interval of finite length; loops and multiple edges between vertices are allowed. While this casual explanation is usually sufficient~\cite{BerKuc13,Mug14}, for our purposes we will need a more formal definition. We are going to follow the approach and formalism in~\cite{Mug19}.

Let $\mE$ be a finite set. Given some $(\ell_\me)_{\me\in\mE}\subset (0,\infty)$, we consider the  disjoint union
\[
\mathcal E:=\bigsqcup\limits_{\me\in\mE} [0,\ell_\me]\ :
\]
we adopt the usual notation $(x,\me)$ for the element of $\mathcal E$ with $x\in [0,\ell_\me]$ and $\me\in\mE$.

Consider the set 
\begin{align*}
\mathcal V:=\bigsqcup\limits_{\me\in\mE}\{0,\ell_\me\}
\end{align*}
of endpoints of $\mathcal E$. 
Given any equivalence relation $\equiv$ on $\mathcal V$, we extend it to an equivalence relation on $\mathcal E$ as follows: two elements $(x_1,\me_1),(x_2,\me_2)\in \mathcal E$ belong to the same equivalence class if and only if $(x_1,\me_1)=(x_2,\me_2)$ or else $(x_1,\me_1),(x_2,\me_2)\in \mathcal V$ and $(x_1,\me_1)\equiv(x_2,\me_2)$; we denote this equivalence relation on $\mathcal E$ again by $\equiv$ and we call $\mathcal G:=\faktor{\mathcal E}{\equiv}$ a \textit{metric graph} and $\mathcal E$, $\mV:=\faktor{\mathcal V}{\equiv}$ its \textit{set of edges} and \textit{of vertices}, respectively. So, a vertex $\mv\in \mV$ is by definition an equivalence class consisting of boundary elements from $\mathcal E$, like $(0,\me)$ or $(\ell_\mf,\mf)$.

Two edges $\me,\mf\in \mE$ are said to be \textit{adjacent} if one endpoint of $\me$ and one endpoint of $\mf$ lie in the same equivalence class $\mv\in\mV$ (i.e., if $\me,\mf$ share an endpoint, up to identification by $\equiv$); in this case we write $\me\sim \mf$. Also, two vertices $\mv,\mw\in \mV$ are said to be \textit{adjacent} if there exists some (not necessarily unique) $\me\in\mE$ such that $\{x,y\}=\{0,\ell_\me\}$ for  representatives $x$ of $\mv$ and $y$ of $\mw$ (i.e., if there is an edge whose endpoints are $\mv,\mw$, up to identification by $\equiv$); in this case we write $\mv\sim \mw$. In either case, with an abuse of notation we also write $\mv\sim \me$.

Let us stress that by definition a metric graph is uniquely determined by a family $(\ell_\me)_{\me\in \mE}$ and an equivalence relation on $\mathcal V$; however, its metric structure is independent on the orientation of the edges!

As a quotient of metric spaces, any metric graph is a metric space in its own right with respect to the canonical quotient metric defined by
\[
d_{\mathcal G}(\xi,\theta):=\inf \sum_{i=1}^k d_{\mathcal E}(\xi_i,\theta_i),\qquad \xi,\theta\in \mathcal G,
\]
where the infimum is taken over all $k\in \mathbb N$ and all pairs of $k$-tuples $(\xi_1,\ldots, \xi_k)$ and $(\theta_1,\ldots, \theta_k)$ with $\xi=\xi_1$, $\theta=\theta_k$, and $\theta_i\sim \xi_{i+1}$ for all $i=1,\ldots,k-1$, \cite[Def.~3.1.12]{BurBurIva01}, where $\sim$ denotes the adjacency relation already introduced in Section~\ref{sez3}. We call $d_{\mathcal G}$ the \textit{path metric} of $\mathcal G$. A metric graph  is said to be \textit{connected} if the path metric  doesn't attain the value $\infty$; in other words, if any two points of $\mathcal G$ can be linked by a path. 
Along with this metric structure there is a natural measure induced by the Lebesgue measure on each interval; accordingly, we can introduce the spaces 
\begin{align*}
C(\mathcal G)\quad\hbox{and}\quad L^2(\mathcal G)
\end{align*}
as well as
\begin{align*}
H^1(\mathcal G):=\{f\in L^2(\mathcal G)\cap C(\mathcal G):f'\in L^2(\mathcal G)\}.
\end{align*}

Again, these definitions do not depend on the orientation of the metric graph; but the notation
\[
f(\xi):=f_\me(x):=f\big((x,\me)\big),\qquad \xi:=(x,\me),
\]
does.

On the graph $\cG$ we aim to introduce a differential operator acting as the second derivative on the functions $f_j(x)$ on every edge $\me_j$; and possibly more general operators of the form 
\[
A_{\max}:=f\mapsto \frac{d}{dx}\left( p\frac{df}{dx}\right)
\]
for some elliptic coefficient $p\ge p_0>0$ of class $W^{1,\infty}$, $p_0\in \mathbb R$.
While it is natural to require that $f_\me \in H^2(0,\ell_\me)$ for every edge $\me$, taking $\bigoplus_{\me\in\mE} H^2(0,\ell_\me)$ as domain only defines an operator acting on functions on $L^2(\mathcal E)$: this is not sufficient in order to define a self-adjoint operator and suitable boundary conditions shall thus be imposed in order for $A_{\max}$ to satisfy the Assumption \ref{hpA}.

Each realization of the elliptic operator $A$ we are interested in is equipped with \textit{natural vertex conditions}: for each element $u$ in its domain
\begin{itemize}
\item $u\in C(\mathcal G)$, and in particular
\begin{equation}
\tag{Cc}
\forall \, \mv \in \mV:\hspace{0,6cm}u_\me(\mv)=u_\mf(\mv), \qquad \text{whenever $\mv \in \me \cap \mf$;} 
\end{equation}
i.e., $u$ is continuous across vertices;
\item $u$ satisfies the \textit{Kirchhoff condition} at each vertex, namely
\begin{equation}
\tag{Kc}
\forall \, \mv \in \mV:\qquad
\sum_{\substack{\me\in\mE\\ (0,\me)\in \mv}} p_{\me}(\mv) u'_\me(\mv)=
\sum_{\substack{\mf\in\mE\\ (\ell_{\mf},\mf)\in \mv}} p_{\mf}(\mv) u'_\mf(\mv),
\end{equation}
i.e., the weighted sum of the inflows equals the weighted sum of the outflows.
\end{itemize}
(Observe that in any vertex with degree 1 the latter becomes a Neumann boundary condition; and that the case $p \equiv 1$ defines the usual Laplacian $\Delta$ with natural vertex conditions on the metric graph $\mathcal G$.)

We can now define  the operator $A$ with natural vertex conditions 
 on $\mathcal{G}$, \ie\
\begin{equation}\label{eq:Ap}
\begin{split}
(A u)_\me(x) &:= (p_{\me}(x) u_\me'(x))',\\
D(A) &:= \left\{u \in C(\cG)\cap \bigoplus_{\me\in\mE} H^2(0,\ell_\me):\ \text{$u$ satisfies (Kc)}\right\}.
\end{split}
\end{equation}

Let us summarize the main results we need in our construction for the operator $A$ with natural vertex conditions. They are part of a general, well-established theory, see \eg\ \cite{Mug14}.

\begin{proposition}\label{prop:lapl-mg-basic}
The operator $A$ with natural vertex conditions on $H = L^2(\cG)$ 
is densely defined, closed, self-adjoint, and negative semi-definite; it has compact resolvent.
\end{proposition}

Thus, $A$ generates a contractive strongly continuous semigroup, denoted by $(e^{t A})_{t \ge 0}$. Hence the abstract Cauchy problem
\begin{equation}
\label{ACP}
\begin{cases}
\tfrac{d}{dt}u(t) = A u(t), &t > 0,\\
u(0)=f,
\end{cases}
\end{equation}
is well-posed: for every $f \in L^2(\cG)$ there exists a unique mild solution given by
$$u(t):=e^{t A}f, \qquad \forall t \ge 0.$$
Moreover, continuous dependence on the initial data holds. Because $A$ is self-adjoint and hence the semigroup is analytic, the solution $u$ is for all $f\in L^2(\mathcal G)$ of class  $C^1((0,\infty);L^2(\mathcal G))\cap C((0,\infty);D(A))$.

By Proposition~\ref{prop:lapl-mg-basic}, the spectrum of $A$ consists of negative eigenvalues of finite multiplicity and the spectral radius satisfies $s(A) = 0 \in \sigma(A)$.
The study of the complete spectrum is still an open problem: actually, only in few cases it is fully determined and in general just some upper and lower bounds on the eigenvalues are known. 
In this work, we are going to emphasize the following property of $\sigma(A)$, see \cite[Theorem 4.3]{KraMugSik07}.

\begin{proposition}\label{propositionbasisnull space}
Let $\mathcal{G} $ be a finite metric graph and denote by $\mathcal G^{(1)},\ldots,\mathcal G^{(l)}$ its connected components. Then, the multiplicity of $0$ as eigenvalue of the operator $A$ with natural vertex conditions is $l$. In particular, the piecewise constant functions $\left\lbrace \mathbbm{1}_h\right\rbrace _{h=1}^{l}$, where
\begin{equation}
\label{orthonormalbasker}
\mathbbm{1}_h(x)=\begin{cases}
1 & \text{if }x \in \mathcal G^{(h)},\\
0 & \text{otherwise,}
\end{cases}
\end{equation}
for all $h=1,\ldots,l$, form a basis of $\ker A$.
\end{proposition}

\subsection{A motivating example}\label{sez:a-mot-ex}

Let us study on the interval $[0,2]$ the heat equation
\[
\left\{
\begin{aligned}
\tfrac{\partial u}{\partial t}(t,x)&=\tfrac{\partial^2 u}{\partial x^2}u(t,x), \qquad  &&t\ge 0,\ x \in [0,2],\\
u(0,x)&= u_0(x), \qquad\qquad  &&x\in [0,2],
\end{aligned}
\right.
\]
where $u_0 \in L^2(0,2)$.
In particular, we are going to analyze two different and well-known boundary value problems: in one case, we  impose two Neumann conditions at $x=0$ and $x=2$, whereas the second setting keeps the same constraints at the boundaries, plus one additional Neumann condition at the middle point $x=1$.\\

\begin{figure}[h]
\begin{center}
\tikzset{every picture/.style={line width=0.75pt}} %set default line width to 0.75pt        
\begin{tikzpicture}[x=0.5pt,y=0.5pt,yscale=-1,xscale=1]
\draw    (320,80) -- (400,80) ;
\draw  [fill={rgb, 255:red, 74; green, 74; blue, 74 }  ,fill opacity=1 ] (237.96,80) .. controls (237.96,78.69) and (239.02,77.63) .. (240.33,77.63) .. controls (241.65,77.63) and (242.71,78.69) .. (242.71,80) .. controls (242.71,81.31) and (241.65,82.38) .. (240.33,82.38) .. controls (239.02,82.38) and (237.96,81.31) .. (237.96,80) -- cycle ;
\draw    (240.33,80) -- (320,80) ;
\draw  [fill={rgb, 255:red, 74; green, 74; blue, 74 }  ,fill opacity=1 ] (397.63,80) .. controls (397.63,78.69) and (398.69,77.63) .. (400,77.63) .. controls (401.31,77.63) and (402.38,78.69) .. (402.38,80) .. controls (402.38,81.31) and (401.31,82.38) .. (400,82.38) .. controls (398.69,82.38) and (397.63,81.31) .. (397.63,80) -- cycle ;
\draw  [fill={rgb, 255:red, 74; green, 74; blue, 74 }  ,fill opacity=1 ] (317.63,80) .. controls (317.63,78.69) and (318.69,77.63) .. (320,77.63) .. controls (321.31,77.63) and (322.38,78.69) .. (322.38,80) .. controls (322.38,81.31) and (321.31,82.38) .. (320,82.38) .. controls (318.69,82.38) and (317.63,81.31) .. (317.63,80) -- cycle ;
\draw (281,66) node [scale=1.2]  {$e_{1}$};
\draw (360.5,66) node [scale=1.2]  {$e_{2}$};
\draw (421,47) node [scale=1.44]  {$\mathcal{G}_1$};
\end{tikzpicture}
\caption{Model A}
\end{center}
\end{figure}

Model A describes the evolution of the heat equation on $[0,2]$ with Neumann boundary conditions in 0 and 2. Formally, however, we consider $[0,2]$ as the graph $\cG_1$ with $\mV = \{0,1,2\}$ and edges $\me_1 = [0,1]$ and $e_2=[1,2]$.

The evolution is thus described by the Laplace operator $\Delta_1$ given by
\begin{align*}
\begin{split}
D(\Delta_1) &= \{u = (u_1, u_2)\,:\, u_i \in H^2(0,1),\ i=1,2, \\
&\qquad\qquad\qquad u_1'(0) = u_2'(2) = 0,\ u_1(1) = u_2(1),\ u_1'(1) - u_2'(1) = 0\},\\
\Delta_1 u&=\frac{d^2 u}{dx^2}.
\end{split}
\end{align*}
The spectrum of $\Delta_1$ clearly agrees with that of the Laplacian with Neumann conditions on $[0,2]$, i.e.,
$$\sigma(\Delta_1)=\left\lbrace \lambda_k = -\frac{k^2\pi^2}{4} , \ k=0,1,2,\ldots\right\rbrace,$$
with associated eigenfunctions
\begin{equation*}
\begin{split}
&e_0(x) =\frac{1}{\sqrt{2}}, \ x \in [0,2],\\
&e_k(x) = \cos \left( \frac{k\pi}{2} x \right), \ x \in [0,2], \qquad k \ge 1. 
\end{split}
\end{equation*}
In this way, for every initial condition $u_0 \in L^2(0,2)$, we can explicitly write the solution in terms of the spectral representation
$$u(t)=e^{t\Delta_1}u_0=\sum_{k=0}^{+\infty} e^{t \lambda_k} (u_0,e_k)_{L^2(0,2)} \ e_k$$
and, as expected, the limit distribution for long times agrees with the average of $u_0$ computed on the interval $[0,2]$
$$\lim_{t \longrightarrow +\infty} u(t)= (u_0,e_0)e_0=P_0 u_0 = \frac{1}{2} \int_{0}^{2} u_0(x) \, {\rm d}x = \dashint_{[0,2]} u_0(x)\,dx.$$

\begin{figure}[h]
\begin{center}
\tikzset{every picture/.style={line width=0.75pt}} %set default line width to 0.75pt        
\begin{tikzpicture}[x=0.5pt,y=0.5pt,yscale=-1,xscale=1]
\draw    (102.51,80.04) -- (177.81,80.04) ;
\draw [shift={(180.16,80.04)}, rotate = 0] [color={rgb, 255:red, 0; green, 0; blue, 0 }  ][line width=0.75]      (0, 0) circle [x radius= 3.35, y radius= 3.35]   ;
\draw [shift={(100.16,80.04)}, rotate = 0] [color={rgb, 255:red, 0; green, 0; blue, 0 }  ][line width=0.75]      (0, 0) circle [x radius= 3.35, y radius= 3.35]   ;
\draw    (182.51,80.04) -- (257.81,80.04) ;
\draw [shift={(260.16,80.04)}, rotate = 0] [color={rgb, 255:red, 0; green, 0; blue, 0 }  ][line width=0.75]      (0, 0) circle [x radius= 3.35, y radius= 3.35]   ;
\draw [shift={(180.16,80.04)}, rotate = 0] [color={rgb, 255:red, 0; green, 0; blue, 0 }  ][line width=0.75]      (0, 0) circle [x radius= 3.35, y radius= 3.35]   ;
\draw    (480.17,80.33) -- (560.17,80.33) ;
\draw  [fill={rgb, 255:red, 74; green, 74; blue, 74 }  ,fill opacity=1 ] (477.79,80.33) .. controls (477.79,79.02) and (478.85,77.96) .. (480.17,77.96) .. controls (481.48,77.96) and (482.54,79.02) .. (482.54,80.33) .. controls (482.54,81.65) and (481.48,82.71) .. (480.17,82.71) .. controls (478.85,82.71) and (477.79,81.65) .. (477.79,80.33) -- cycle ;
\draw  [fill={rgb, 255:red, 74; green, 74; blue, 74 }  ,fill opacity=1 ] (557.79,80.33) .. controls (557.79,79.02) and (558.85,77.96) .. (560.17,77.96) .. controls (561.48,77.96) and (562.54,79.02) .. (562.54,80.33) .. controls (562.54,81.65) and (561.48,82.71) .. (560.17,82.71) .. controls (558.85,82.71) and (557.79,81.65) .. (557.79,80.33) -- cycle ;
\draw    (380.5,80) -- (460.17,80) ;
\draw  [fill={rgb, 255:red, 74; green, 74; blue, 74 }  ,fill opacity=1 ] (378.13,80) .. controls (378.13,78.69) and (379.19,77.63) .. (380.5,77.63) .. controls (381.81,77.63) and (382.88,78.69) .. (382.88,80) .. controls (382.88,81.31) and (381.81,82.38) .. (380.5,82.38) .. controls (379.19,82.38) and (378.13,81.31) .. (378.13,80) -- cycle ;
\draw  [fill={rgb, 255:red, 74; green, 74; blue, 74 }  ,fill opacity=1 ] (457.79,80) .. controls (457.79,78.69) and (458.85,77.63) .. (460.17,77.63) .. controls (461.48,77.63) and (462.54,78.69) .. (462.54,80) .. controls (462.54,81.31) and (461.48,82.38) .. (460.17,82.38) .. controls (458.85,82.38) and (457.79,81.31) .. (457.79,80) -- cycle ;
\draw [color={rgb, 255:red, 208; green, 2; blue, 27 }  ,draw opacity=1 ]   (300.13,80.25) -- (338.63,80.01) ;
\draw [shift={(340.63,80)}, rotate = 539.65] [color={rgb, 255:red, 208; green, 2; blue, 27 }  ,draw opacity=1 ][line width=0.75]    (10.93,-3.29) .. controls (6.95,-1.4) and (3.31,-0.3) .. (0,0) .. controls (3.31,0.3) and (6.95,1.4) .. (10.93,3.29)   ;
\draw (101.33,61.33) node  [align=left] {\textbf{{\large N}}};
\draw (181.67,62) node  [align=left] {\textbf{{\large N}}};
\draw (261.67,61.67) node  [align=left] {\textbf{{\large N}}};
\draw (101.3,93.5) node   {$0$};
\draw (261,93.2) node   {$2$};
\draw (181,94) node   {$1$};
\draw (423,66) node [scale=1.2]  {$e_{1}$};
\draw (520,66) node [scale=1.2]  {$e_{2}$};
\draw (381.3,94) node   {$0$};
\draw (460,93) node   {$1$};
\draw (561,93) node   {$2$};
\draw (480.8,93) node   {$1'$};
\draw (569,43) node [scale=1.44]  {$\mathcal{G}_2$};
\end{tikzpicture}
\caption{Model B: on the right, the correct interpretation as a network equation with a disconnected graph}
\end{center}
\end{figure}
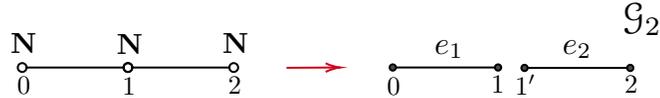

Model B describes the evolution of the heat equation on $[0,2]$ with Neumann boundary conditions in 0, in 2, \textit{as well as} in 1. Formally, we consider $[0,2]$ as the graph $\cG_2$ with $\mV = \{0,1, 1', 2\}$ and edges $\me_1 = [0,1]$ and $e_2=[1',2]$.

The Laplace operator associated with $\cG_2$ is $\Delta_2$ with domain
\begin{equation*}
\begin{split}
D(\Delta_2) &= \{u = (u_1, u_2)\,:\, u_i \in H^2(0,1),\ i=1,2, \ u_1'(0) = u_1'(1) = 0, \ u_2'(0) = u_2'(1)=0\},\\
\Delta_2 u&=\frac{d^2 u}{dx^2}.
\end{split}
\end{equation*}

Here the dynamics is somehow different from the previous one: in fact, the Neumann condition placed in $x=1$ acts like an insulating ``wall" through which heat exchanges are not allowed. 
The spectrum in this case is
$$\sigma({\Delta}_2)= \left\lbrace \overline{\mu}_k = -k^2\pi^2, \ k=0,1,2,\ldots\right\rbrace,$$
where every eigenvalue has now multiplicity two. 
\\
For every initial condition $g = (g_1, g_2) \in L^2(\mathcal{G})$ the solution $u(t)$ converges, as $t \to \infty$, to 
$\displaystyle (\dashint g_1, \dashint g_2)$.

Starting from these two models, we now introduce the following scenario: imagine that we are going to study the heat diffusion along the interval $[0,2]$ with Neumann boundary conditions. However,  at each renewal time $T_n$ we can decide to add or remove one third Neumann condition at $x=1$. 
In particular, the choice of considering three or two constraints is determined by a suitable random process.
This means that the system switches between Model A and Model B and the stochastic evolution problem is of the form \eqref{CPS}.

We shall see that the asymptotic behavior of our systems is given by the uniform $\bP$-almost sure convergence towards the orthogonal projector $P_0$ to the constant functions.

\subsection{The general model}

Like in Section~\ref{sez3}, we are going to introduce ensembles of metric graphs.

\begin{assumption}\label{a-metgraphs}
 $\cC = \{\mathcal G_1, . . . , \mathcal G_N\}$, 
 where $\mathcal G_1, . . . , \mathcal G_N$ are metric graphs with the same edge set $\mathcal E$ (i.e., defined upon the same finite set $\mE$ and the same vector $(\ell_\me)_{\me\in\mE}$) but possibly different sets of vertices $\mathcal V_1:=\mathcal V(\mathcal G_1), . . . , \mathcal V_N:=\mathcal V(\mathcal G_N)$ (i.e., the equivalence relations $\equiv_1,\ldots, \equiv_N$ may be different).
 \end{assumption}

Once again, we introduce a probability space $(\Omega,\mathcal{F},\mathbb{P})$ and a semi-Markov process $(Z(t))_{t\ge 0}$ satisfying Assumption \ref{hpZ}.

At this point, we can associate with each graph $\mathcal G_i$ in $\mathcal{C}$
an operator $A_i$  with natural vertex conditions and {elliptic coefficient $p_i \in W^{1,\infty}$ as in~\eqref{eq:Ap}, which we denote by
$$(A_i, D(A_i)),\qquad i=1,\ldots,N:$$
we emphasize that the different vertex sets induce different operator domains, even though all operators satisfy the same class of vertex conditions: for example, ``cutting through a vertex'', hence producing two vertices of lower degree out of a vertex of larger degree, induces a new operator with relaxed continuity conditions (and two new Kirchhoff conditions).

By Proposition~\ref{prop:lapl-mg-basic}, all these operators satisfy the Assumptions \ref{hpA} and \ref{hpc}.
We can state our main problem, i.e., the continuous random evolution  on metric graphs
\begin{equation}
\label{CRE}
\begin{cases}
\tfrac{d}{d t}u(t) = A_{X_k}u(t), \hspace{0,5cm}t \in [T_k,T_{k+1}), \\
u(0)=f \in  L^2(\cG).
\end{cases}
\end{equation}
We recall that $S(t)$ is the random propagator associated with problem \eqref{CRE} such that $u(t) = S(t)f$.
Our interest is again to prove a link between the convergence of $S(t)$ towards the orthogonal projector $P_0$ with the connectedness of the union of the graphs in $\cC$.
However, the key point here is to give a definition of the concept of {\em union graph}
in the metric setting: this follows immediately from the above formalism, see~\cite{Mug19}.
 
\begin{definition}[Union and intersection of metric graphs]
\label{defunionegrafimetrici}
Let $\mathcal{G}_1,\ldots,\mathcal{G}_N$ be metric graphs defined on the same $\mathcal E$, i.e., $\mathcal G_i=\faktor{\mathcal E}{\equiv_i}$, $i=1,\ldots,N$.
Denote by $\equiv_\cup$ and by $\equiv_\cap$ the equivalence relations obtained by taking the  reflexive, symmetric, and transitive closure of $\bigcup_{i=1}^N \equiv_i \ \subset \mathcal V\times \mathcal V$ and $\bigcap_{i=1}^N \equiv_i \ \subset \mathcal V\times \mathcal V$, respectively.
Then, we call \emph{union} and \emph{intersection metric graph} the metric graphs 
$$\mathcal{G}_\cup:=\faktor{\mathcal E}{\mathcal \equiv_\cup}\qquad\hbox{and}\qquad \mathcal{G}_\cap:=\faktor{\mathcal E}{\mathcal \equiv_\cap},$$
respectively.
\end{definition}

In Fig. \ref{EXAMPLE} we can consider some examples of union graphs. 
\vspace{0,5cm}

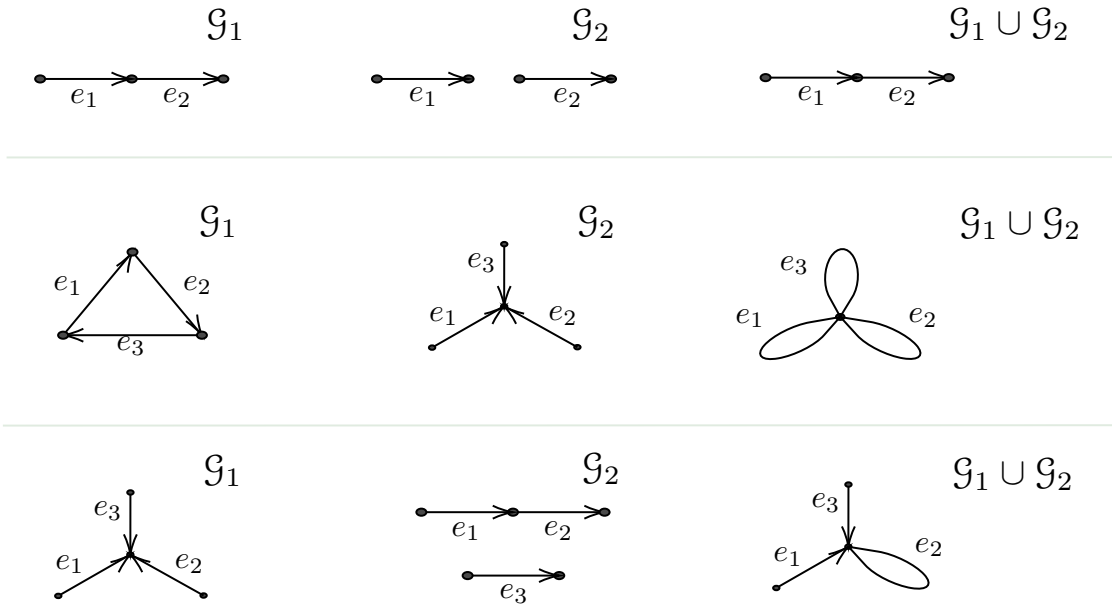
\begin{figure}[h]
\begin{center}
\tikzset{every picture/.style={line width=0.75pt}} %set default line width to 0.75pt        
\begin{tikzpicture}[x=0.7pt,y=0.5pt,yscale=-1,xscale=1]
\draw    (59.83,269.93) -- (132.83,269.93) ;
\draw [shift={(57.83,269.93)}, rotate = 0] [color={rgb, 255:red, 0; green, 0; blue, 0 }  ][line width=0.75]    (10.93,-4.9) .. controls (6.95,-2.3) and (3.31,-0.67) .. (0,0) .. controls (3.31,0.67) and (6.95,2.3) .. (10.93,4.9)   ;
\draw    (131.81,268.21) -- (95.33,206.93) ;
\draw [shift={(132.83,269.93)}, rotate = 239.24] [color={rgb, 255:red, 0; green, 0; blue, 0 }  ][line width=0.75]    (10.93,-4.9) .. controls (6.95,-2.3) and (3.31,-0.67) .. (0,0) .. controls (3.31,0.67) and (6.95,2.3) .. (10.93,4.9)   ;
\draw    (94.31,208.65) -- (57.83,269.93) ;
\draw [shift={(95.33,206.93)}, rotate = 120.76] [color={rgb, 255:red, 0; green, 0; blue, 0 }  ][line width=0.75]    (10.93,-4.9) .. controls (6.95,-2.3) and (3.31,-0.67) .. (0,0) .. controls (3.31,0.67) and (6.95,2.3) .. (10.93,4.9)   ;
\draw  [fill={rgb, 255:red, 74; green, 74; blue, 74 }  ,fill opacity=1 ] (55.13,269.93) .. controls (55.13,268.44) and (56.34,267.23) .. (57.83,267.23) .. controls (59.32,267.23) and (60.53,268.44) .. (60.53,269.93) .. controls (60.53,271.42) and (59.32,272.63) .. (57.83,272.63) .. controls (56.34,272.63) and (55.13,271.42) .. (55.13,269.93) -- cycle ;
\draw  [fill={rgb, 255:red, 74; green, 74; blue, 74 }  ,fill opacity=1 ] (92.63,206.93) .. controls (92.63,205.44) and (93.84,204.23) .. (95.33,204.23) .. controls (96.82,204.23) and (98.03,205.44) .. (98.03,206.93) .. controls (98.03,208.42) and (96.82,209.63) .. (95.33,209.63) .. controls (93.84,209.63) and (92.63,208.42) .. (92.63,206.93) -- cycle ;
\draw  [fill={rgb, 255:red, 74; green, 74; blue, 74 }  ,fill opacity=1 ] (130.13,269.93) .. controls (130.13,268.44) and (131.34,267.23) .. (132.83,267.23) .. controls (134.32,267.23) and (135.53,268.44) .. (135.53,269.93) .. controls (135.53,271.42) and (134.32,272.63) .. (132.83,272.63) .. controls (131.34,272.63) and (130.13,271.42) .. (130.13,269.93) -- cycle ;
\draw    (296.16,246.03) -- (296.2,200.99) ;
\draw [shift={(296.16,248.03)}, rotate = 270.05] [color={rgb, 255:red, 0; green, 0; blue, 0 }  ][line width=0.75]    (10.93,-4.9) .. controls (6.95,-2.3) and (3.31,-0.67) .. (0,0) .. controls (3.31,0.67) and (6.95,2.3) .. (10.93,4.9)   ;
\draw  [fill={rgb, 255:red, 74; green, 74; blue, 74 }  ,fill opacity=1 ] (294.48,200.99) .. controls (294.48,200.05) and (295.25,199.3) .. (296.2,199.3) .. controls (297.15,199.3) and (297.92,200.05) .. (297.92,200.99) .. controls (297.92,201.92) and (297.15,202.67) .. (296.2,202.67) .. controls (295.25,202.67) and (294.48,201.92) .. (294.48,200.99) -- cycle ;
\draw  [fill={rgb, 255:red, 74; green, 74; blue, 74 }  ,fill opacity=1 ] (294.44,248.03) .. controls (294.44,247.1) and (295.21,246.35) .. (296.16,246.35) .. controls (297.11,246.35) and (297.87,247.1) .. (297.87,248.03) .. controls (297.87,248.97) and (297.11,249.72) .. (296.16,249.72) .. controls (295.21,249.72) and (294.44,248.97) .. (294.44,248.03) -- cycle ;
\draw    (257.22,279.31) -- (294.6,249.29) ;
\draw [shift={(296.16,248.03)}, rotate = 501.23] [color={rgb, 255:red, 0; green, 0; blue, 0 }  ][line width=0.75]    (10.93,-4.9) .. controls (6.95,-2.3) and (3.31,-0.67) .. (0,0) .. controls (3.31,0.67) and (6.95,2.3) .. (10.93,4.9)   ;
\draw  [fill={rgb, 255:red, 74; green, 74; blue, 74 }  ,fill opacity=1 ] (334.02,279) .. controls (334.02,278.07) and (334.79,277.32) .. (335.73,277.32) .. controls (336.68,277.32) and (337.45,278.07) .. (337.45,279) .. controls (337.45,279.93) and (336.68,280.69) .. (335.73,280.69) .. controls (334.79,280.69) and (334.02,279.93) .. (334.02,279) -- cycle ;
\draw  [fill={rgb, 255:red, 74; green, 74; blue, 74 }  ,fill opacity=1 ] (255.5,279.31) .. controls (255.5,278.38) and (256.27,277.63) .. (257.22,277.63) .. controls (258.16,277.63) and (258.93,278.38) .. (258.93,279.31) .. controls (258.93,280.25) and (258.16,281) .. (257.22,281) .. controls (256.27,281) and (255.5,280.25) .. (255.5,279.31) -- cycle ;
\draw    (297.73,249.27) -- (335.73,279) ;
\draw [shift={(296.16,248.03)}, rotate = 38.04] [color={rgb, 255:red, 0; green, 0; blue, 0 }  ][line width=0.75]    (10.93,-4.9) .. controls (6.95,-2.3) and (3.31,-0.67) .. (0,0) .. controls (3.31,0.67) and (6.95,2.3) .. (10.93,4.9)   ;
\draw  [fill={rgb, 255:red, 74; green, 74; blue, 74 }  ,fill opacity=1 ] (475.46,256.1) .. controls (475.46,254.85) and (476.49,253.85) .. (477.76,253.85) .. controls (479.03,253.85) and (480.06,254.85) .. (480.06,256.1) .. controls (480.06,257.34) and (479.03,258.35) .. (477.76,258.35) .. controls (476.49,258.35) and (475.46,257.34) .. (475.46,256.1) -- cycle ;
\draw   (471.85,240.95) .. controls (471.85,240.95) and (471.85,240.95) .. (471.85,240.95) .. controls (471.85,240.95) and (471.85,240.95) .. (471.85,240.95) .. controls (468.58,232.69) and (468.89,219.27) .. (472.53,210.97) .. controls (476.17,202.67) and (481.77,202.63) .. (485.04,210.89) .. controls (488.31,219.15) and (488,232.57) .. (484.36,240.88) .. controls (479.96,250.9) and (477.76,255.97) .. (477.76,256.1) .. controls (477.76,255.97) and (475.79,250.92) .. (471.85,240.95) -- cycle ;
\draw   (493.85,259.82) .. controls (502.65,261.82) and (513.79,269.61) .. (518.74,277.24) .. controls (523.69,284.86) and (520.57,289.42) .. (511.77,287.43) .. controls (502.97,285.44) and (491.83,277.64) .. (486.88,270.02) .. controls (480.9,260.81) and (477.87,256.18) .. (477.76,256.1) .. controls (477.87,256.18) and (483.23,257.41) .. (493.85,259.82) -- cycle ;
\draw   (469.02,269.87) .. controls (469.02,269.87) and (469.02,269.87) .. (469.02,269.87) .. controls (469.02,269.87) and (469.02,269.87) .. (469.02,269.87) .. controls (464.29,277.41) and (453.2,285.28) .. (444.25,287.45) .. controls (435.31,289.61) and (431.9,285.25) .. (436.63,277.71) .. controls (441.37,270.17) and (452.46,262.3) .. (461.41,260.13) .. controls (461.41,260.13) and (461.41,260.13) .. (461.41,260.13) .. controls (472.21,257.51) and (477.66,256.17) .. (477.76,256.1) .. controls (477.66,256.17) and (474.74,260.77) .. (469.02,269.87) -- cycle ;
\draw [color={rgb, 255:red, 224; green, 235; blue, 224 }  ,draw opacity=0.99 ][line width=0.75]    (25.29,338.36) -- (625.57,338.07) ;
\draw [color={rgb, 255:red, 224; green, 235; blue, 224 }  ,draw opacity=0.99 ][line width=0.75]    (27.29,135.36) -- (627.57,135.07) ;
\draw    (94.92,75.92) -- (142.42,75.92) ;
\draw [shift={(144.42,75.92)}, rotate = 180] [color={rgb, 255:red, 0; green, 0; blue, 0 }  ][line width=0.75]    (10.93,-4.9) .. controls (6.95,-2.3) and (3.31,-0.67) .. (0,0) .. controls (3.31,0.67) and (6.95,2.3) .. (10.93,4.9)   ;
\draw  [fill={rgb, 255:red, 74; green, 74; blue, 74 }  ,fill opacity=1 ] (92.22,75.92) .. controls (92.22,74.43) and (93.43,73.22) .. (94.92,73.22) .. controls (96.41,73.22) and (97.62,74.43) .. (97.62,75.92) .. controls (97.62,77.41) and (96.41,78.62) .. (94.92,78.62) .. controls (93.43,78.62) and (92.22,77.41) .. (92.22,75.92) -- cycle ;
\draw  [fill={rgb, 255:red, 74; green, 74; blue, 74 }  ,fill opacity=1 ] (141.72,75.92) .. controls (141.72,74.43) and (142.93,73.22) .. (144.42,73.22) .. controls (145.91,73.22) and (147.12,74.43) .. (147.12,75.92) .. controls (147.12,77.41) and (145.91,78.62) .. (144.42,78.62) .. controls (142.93,78.62) and (141.72,77.41) .. (141.72,75.92) -- cycle ;
\draw    (45.42,75.92) -- (92.92,75.92) ;
\draw [shift={(94.92,75.92)}, rotate = 180] [color={rgb, 255:red, 0; green, 0; blue, 0 }  ][line width=0.75]    (10.93,-4.9) .. controls (6.95,-2.3) and (3.31,-0.67) .. (0,0) .. controls (3.31,0.67) and (6.95,2.3) .. (10.93,4.9)   ;
\draw  [fill={rgb, 255:red, 74; green, 74; blue, 74 }  ,fill opacity=1 ] (42.72,75.92) .. controls (42.72,74.43) and (43.93,73.22) .. (45.42,73.22) .. controls (46.91,73.22) and (48.12,74.43) .. (48.12,75.92) .. controls (48.12,77.41) and (46.91,78.62) .. (45.42,78.62) .. controls (43.93,78.62) and (42.72,77.41) .. (42.72,75.92) -- cycle ;
\draw  [fill={rgb, 255:red, 74; green, 74; blue, 74 }  ,fill opacity=1 ] (274.22,75.92) .. controls (274.22,74.43) and (275.43,73.22) .. (276.92,73.22) .. controls (278.41,73.22) and (279.62,74.43) .. (279.62,75.92) .. controls (279.62,77.41) and (278.41,78.62) .. (276.92,78.62) .. controls (275.43,78.62) and (274.22,77.41) .. (274.22,75.92) -- cycle ;
\draw    (227.42,75.92) -- (274.92,75.92) ;
\draw [shift={(276.92,75.92)}, rotate = 180] [color={rgb, 255:red, 0; green, 0; blue, 0 }  ][line width=0.75]    (10.93,-4.9) .. controls (6.95,-2.3) and (3.31,-0.67) .. (0,0) .. controls (3.31,0.67) and (6.95,2.3) .. (10.93,4.9)   ;
\draw  [fill={rgb, 255:red, 74; green, 74; blue, 74 }  ,fill opacity=1 ] (224.72,75.92) .. controls (224.72,74.43) and (225.93,73.22) .. (227.42,73.22) .. controls (228.91,73.22) and (230.12,74.43) .. (230.12,75.92) .. controls (230.12,77.41) and (228.91,78.62) .. (227.42,78.62) .. controls (225.93,78.62) and (224.72,77.41) .. (224.72,75.92) -- cycle ;
\draw  [fill={rgb, 255:red, 74; green, 74; blue, 74 }  ,fill opacity=1 ] (351.22,75.92) .. controls (351.22,74.43) and (352.43,73.22) .. (353.92,73.22) .. controls (355.41,73.22) and (356.62,74.43) .. (356.62,75.92) .. controls (356.62,77.41) and (355.41,78.62) .. (353.92,78.62) .. controls (352.43,78.62) and (351.22,77.41) .. (351.22,75.92) -- cycle ;
\draw    (304.42,75.92) -- (351.92,75.92) ;
\draw [shift={(353.92,75.92)}, rotate = 180] [color={rgb, 255:red, 0; green, 0; blue, 0 }  ][line width=0.75]    (10.93,-4.9) .. controls (6.95,-2.3) and (3.31,-0.67) .. (0,0) .. controls (3.31,0.67) and (6.95,2.3) .. (10.93,4.9)   ;
\draw  [fill={rgb, 255:red, 74; green, 74; blue, 74 }  ,fill opacity=1 ] (301.72,75.92) .. controls (301.72,74.43) and (302.93,73.22) .. (304.42,73.22) .. controls (305.91,73.22) and (307.12,74.43) .. (307.12,75.92) .. controls (307.12,77.41) and (305.91,78.62) .. (304.42,78.62) .. controls (302.93,78.62) and (301.72,77.41) .. (301.72,75.92) -- cycle ;
\draw    (486.92,74.92) -- (534.42,74.92) ;
\draw [shift={(536.42,74.92)}, rotate = 180] [color={rgb, 255:red, 0; green, 0; blue, 0 }  ][line width=0.75]    (10.93,-4.9) .. controls (6.95,-2.3) and (3.31,-0.67) .. (0,0) .. controls (3.31,0.67) and (6.95,2.3) .. (10.93,4.9)   ;
\draw  [fill={rgb, 255:red, 74; green, 74; blue, 74 }  ,fill opacity=1 ] (484.22,74.92) .. controls (484.22,73.43) and (485.43,72.22) .. (486.92,72.22) .. controls (488.41,72.22) and (489.62,73.43) .. (489.62,74.92) .. controls (489.62,76.41) and (488.41,77.62) .. (486.92,77.62) .. controls (485.43,77.62) and (484.22,76.41) .. (484.22,74.92) -- cycle ;
\draw  [fill={rgb, 255:red, 74; green, 74; blue, 74 }  ,fill opacity=1 ] (533.72,74.92) .. controls (533.72,73.43) and (534.93,72.22) .. (536.42,72.22) .. controls (537.91,72.22) and (539.12,73.43) .. (539.12,74.92) .. controls (539.12,76.41) and (537.91,77.62) .. (536.42,77.62) .. controls (534.93,77.62) and (533.72,76.41) .. (533.72,74.92) -- cycle ;
\draw    (437.42,74.92) -- (484.92,74.92) ;
\draw [shift={(486.92,74.92)}, rotate = 180] [color={rgb, 255:red, 0; green, 0; blue, 0 }  ][line width=0.75]    (10.93,-4.9) .. controls (6.95,-2.3) and (3.31,-0.67) .. (0,0) .. controls (3.31,0.67) and (6.95,2.3) .. (10.93,4.9)   ;
\draw  [fill={rgb, 255:red, 74; green, 74; blue, 74 }  ,fill opacity=1 ] (434.72,74.92) .. controls (434.72,73.43) and (435.93,72.22) .. (437.42,72.22) .. controls (438.91,72.22) and (440.12,73.43) .. (440.12,74.92) .. controls (440.12,76.41) and (438.91,77.62) .. (437.42,77.62) .. controls (435.93,77.62) and (434.72,76.41) .. (434.72,74.92) -- cycle ;
\draw    (94.16,434.03) -- (94.2,388.99) ;
\draw [shift={(94.16,436.03)}, rotate = 270.05] [color={rgb, 255:red, 0; green, 0; blue, 0 }  ][line width=0.75]    (10.93,-4.9) .. controls (6.95,-2.3) and (3.31,-0.67) .. (0,0) .. controls (3.31,0.67) and (6.95,2.3) .. (10.93,4.9)   ;
\draw  [fill={rgb, 255:red, 74; green, 74; blue, 74 }  ,fill opacity=1 ] (92.48,388.99) .. controls (92.48,388.05) and (93.25,387.3) .. (94.2,387.3) .. controls (95.15,387.3) and (95.92,388.05) .. (95.92,388.99) .. controls (95.92,389.92) and (95.15,390.67) .. (94.2,390.67) .. controls (93.25,390.67) and (92.48,389.92) .. (92.48,388.99) -- cycle ;
\draw  [fill={rgb, 255:red, 74; green, 74; blue, 74 }  ,fill opacity=1 ] (92.44,436.03) .. controls (92.44,435.1) and (93.21,434.35) .. (94.16,434.35) .. controls (95.11,434.35) and (95.87,435.1) .. (95.87,436.03) .. controls (95.87,436.97) and (95.11,437.72) .. (94.16,437.72) .. controls (93.21,437.72) and (92.44,436.97) .. (92.44,436.03) -- cycle ;
\draw    (55.22,467.31) -- (92.6,437.29) ;
\draw [shift={(94.16,436.03)}, rotate = 501.23] [color={rgb, 255:red, 0; green, 0; blue, 0 }  ][line width=0.75]    (10.93,-4.9) .. controls (6.95,-2.3) and (3.31,-0.67) .. (0,0) .. controls (3.31,0.67) and (6.95,2.3) .. (10.93,4.9)   ;
\draw  [fill={rgb, 255:red, 74; green, 74; blue, 74 }  ,fill opacity=1 ] (132.02,467) .. controls (132.02,466.07) and (132.79,465.32) .. (133.73,465.32) .. controls (134.68,465.32) and (135.45,466.07) .. (135.45,467) .. controls (135.45,467.93) and (134.68,468.69) .. (133.73,468.69) .. controls (132.79,468.69) and (132.02,467.93) .. (132.02,467) -- cycle ;
\draw  [fill={rgb, 255:red, 74; green, 74; blue, 74 }  ,fill opacity=1 ] (53.5,467.31) .. controls (53.5,466.38) and (54.27,465.63) .. (55.22,465.63) .. controls (56.16,465.63) and (56.93,466.38) .. (56.93,467.31) .. controls (56.93,468.25) and (56.16,469) .. (55.22,469) .. controls (54.27,469) and (53.5,468.25) .. (53.5,467.31) -- cycle ;
\draw    (95.73,437.27) -- (133.73,467) ;
\draw [shift={(94.16,436.03)}, rotate = 38.04] [color={rgb, 255:red, 0; green, 0; blue, 0 }  ][line width=0.75]    (10.93,-4.9) .. controls (6.95,-2.3) and (3.31,-0.67) .. (0,0) .. controls (3.31,0.67) and (6.95,2.3) .. (10.93,4.9)   ;
\draw    (300.92,403.92) -- (348.42,403.92) ;
\draw [shift={(350.42,403.92)}, rotate = 180] [color={rgb, 255:red, 0; green, 0; blue, 0 }  ][line width=0.75]    (10.93,-4.9) .. controls (6.95,-2.3) and (3.31,-0.67) .. (0,0) .. controls (3.31,0.67) and (6.95,2.3) .. (10.93,4.9)   ;
\draw  [fill={rgb, 255:red, 74; green, 74; blue, 74 }  ,fill opacity=1 ] (298.22,403.92) .. controls (298.22,402.43) and (299.43,401.22) .. (300.92,401.22) .. controls (302.41,401.22) and (303.62,402.43) .. (303.62,403.92) .. controls (303.62,405.41) and (302.41,406.62) .. (300.92,406.62) .. controls (299.43,406.62) and (298.22,405.41) .. (298.22,403.92) -- cycle ;
\draw  [fill={rgb, 255:red, 74; green, 74; blue, 74 }  ,fill opacity=1 ] (347.72,403.92) .. controls (347.72,402.43) and (348.93,401.22) .. (350.42,401.22) .. controls (351.91,401.22) and (353.12,402.43) .. (353.12,403.92) .. controls (353.12,405.41) and (351.91,406.62) .. (350.42,406.62) .. controls (348.93,406.62) and (347.72,405.41) .. (347.72,403.92) -- cycle ;
\draw    (251.42,403.92) -- (298.92,403.92) ;
\draw [shift={(300.92,403.92)}, rotate = 180] [color={rgb, 255:red, 0; green, 0; blue, 0 }  ][line width=0.75]    (10.93,-4.9) .. controls (6.95,-2.3) and (3.31,-0.67) .. (0,0) .. controls (3.31,0.67) and (6.95,2.3) .. (10.93,4.9)   ;
\draw  [fill={rgb, 255:red, 74; green, 74; blue, 74 }  ,fill opacity=1 ] (248.72,403.92) .. controls (248.72,402.43) and (249.93,401.22) .. (251.42,401.22) .. controls (252.91,401.22) and (254.12,402.43) .. (254.12,403.92) .. controls (254.12,405.41) and (252.91,406.62) .. (251.42,406.62) .. controls (249.93,406.62) and (248.72,405.41) .. (248.72,403.92) -- cycle ;
\draw  [fill={rgb, 255:red, 74; green, 74; blue, 74 }  ,fill opacity=1 ] (323.22,451.92) .. controls (323.22,450.43) and (324.43,449.22) .. (325.92,449.22) .. controls (327.41,449.22) and (328.62,450.43) .. (328.62,451.92) .. controls (328.62,453.41) and (327.41,454.62) .. (325.92,454.62) .. controls (324.43,454.62) and (323.22,453.41) .. (323.22,451.92) -- cycle ;
\draw    (276.42,451.92) -- (323.92,451.92) ;
\draw [shift={(325.92,451.92)}, rotate = 180] [color={rgb, 255:red, 0; green, 0; blue, 0 }  ][line width=0.75]    (10.93,-4.9) .. controls (6.95,-2.3) and (3.31,-0.67) .. (0,0) .. controls (3.31,0.67) and (6.95,2.3) .. (10.93,4.9)   ;
\draw  [fill={rgb, 255:red, 74; green, 74; blue, 74 }  ,fill opacity=1 ] (273.72,451.92) .. controls (273.72,450.43) and (274.93,449.22) .. (276.42,449.22) .. controls (277.91,449.22) and (279.12,450.43) .. (279.12,451.92) .. controls (279.12,453.41) and (277.91,454.62) .. (276.42,454.62) .. controls (274.93,454.62) and (273.72,453.41) .. (273.72,451.92) -- cycle ;
\draw    (482.16,428.03) -- (482.2,382.99) ;
\draw [shift={(482.16,430.03)}, rotate = 270.05] [color={rgb, 255:red, 0; green, 0; blue, 0 }  ][line width=0.75]    (10.93,-4.9) .. controls (6.95,-2.3) and (3.31,-0.67) .. (0,0) .. controls (3.31,0.67) and (6.95,2.3) .. (10.93,4.9)   ;
\draw  [fill={rgb, 255:red, 74; green, 74; blue, 74 }  ,fill opacity=1 ] (480.48,382.99) .. controls (480.48,382.05) and (481.25,381.3) .. (482.2,381.3) .. controls (483.15,381.3) and (483.92,382.05) .. (483.92,382.99) .. controls (483.92,383.92) and (483.15,384.67) .. (482.2,384.67) .. controls (481.25,384.67) and (480.48,383.92) .. (480.48,382.99) -- cycle ;
\draw  [fill={rgb, 255:red, 74; green, 74; blue, 74 }  ,fill opacity=1 ] (480.44,430.03) .. controls (480.44,429.1) and (481.21,428.35) .. (482.16,428.35) .. controls (483.11,428.35) and (483.87,429.1) .. (483.87,430.03) .. controls (483.87,430.97) and (483.11,431.72) .. (482.16,431.72) .. controls (481.21,431.72) and (480.44,430.97) .. (480.44,430.03) -- cycle ;
\draw    (443.22,461.31) -- (480.6,431.29) ;
\draw [shift={(482.16,430.03)}, rotate = 501.23] [color={rgb, 255:red, 0; green, 0; blue, 0 }  ][line width=0.75]    (10.93,-4.9) .. controls (6.95,-2.3) and (3.31,-0.67) .. (0,0) .. controls (3.31,0.67) and (6.95,2.3) .. (10.93,4.9)   ;
\draw  [fill={rgb, 255:red, 74; green, 74; blue, 74 }  ,fill opacity=1 ] (441.5,461.31) .. controls (441.5,460.38) and (442.27,459.63) .. (443.22,459.63) .. controls (444.16,459.63) and (444.93,460.38) .. (444.93,461.31) .. controls (444.93,462.25) and (444.16,463) .. (443.22,463) .. controls (442.27,463) and (441.5,462.25) .. (441.5,461.31) -- cycle ;
\draw   (498.25,433.76) .. controls (507.05,435.75) and (518.19,443.55) .. (523.14,451.17) .. controls (528.09,458.79) and (524.97,463.36) .. (516.17,461.36) .. controls (507.37,459.37) and (496.23,451.57) .. (491.28,443.95) .. controls (485.3,434.75) and (482.26,430.11) .. (482.16,430.03) .. controls (482.26,430.11) and (487.63,431.35) .. (498.25,433.76) -- cycle ;
\draw (141.67,183.67) node [scale=1.44]  {$\mathcal{G}_{1}$};
\draw (346,184) node [scale=1.44]  {$\mathcal{G}_{2}$};
\draw (575,185) node [scale=1.44]  {$\mathcal{G}_{1} \cup \mathcal{G}_{2}$};
\draw (60.5,232.5) node [scale=1.2]  {$e_{1}$};
\draw (263.5,255.5) node [scale=1.2]  {$e_{1}$};
\draw (429,257) node [scale=1.2]  {$e_{1}$};
\draw (130.5,232) node [scale=1.2]  {$e_{2}$};
\draw (328,255) node [scale=1.2]  {$e_{2}$};
\draw (522.5,256.5) node [scale=1.2]  {$e_{2}$};
\draw (94.5,278) node [scale=1.2]  {$e_{3}$};
\draw (284,215) node [scale=1.2]  {$e_{3}$};
\draw (453,217) node [scale=1.2]  {$e_{3}$};
\draw (145.67,34.67) node [scale=1.44]  {$\mathcal{G}_{1}$};
\draw (343,35) node [scale=1.44]  {$\mathcal{G}_{2}$};
\draw (569,34) node [scale=1.44]  {$\mathcal{G}_{1} \cup \mathcal{G}_{2}$};
\draw (69.5,90.5) node [scale=1.2]  {$e_{1}$};
\draw (252.5,89.5) node [scale=1.2]  {$e_{1}$};
\draw (462.5,88.5) node [scale=1.2]  {$e_{1}$};
\draw (119.5,90) node [scale=1.2]  {$e_{2}$};
\draw (330.5,90) node [scale=1.2]  {$e_{2}$};
\draw (511.5,89) node [scale=1.2]  {$e_{2}$};
\draw (144,372) node [scale=1.44]  {$\mathcal{G}_{1}$};
\draw (61.5,443.5) node [scale=1.2]  {$e_{1}$};
\draw (126,443) node [scale=1.2]  {$e_{2}$};
\draw (82,403) node [scale=1.2]  {$e_{3}$};
\draw (348.67,370.67) node [scale=1.44]  {$\mathcal{G}_{2}$};
\draw (275.5,418.5) node [scale=1.2]  {$e_{1}$};
\draw (325.5,418) node [scale=1.2]  {$e_{2}$};
\draw (301.5,465.5) node [scale=1.2]  {$e_{3}$};
\draw (449.5,437.5) node [scale=1.2]  {$e_{1}$};
\draw (470,397) node [scale=1.2]  {$e_{3}$};
\draw (525.9,430.5) node [scale=1.2]  {$e_{2}$};
\draw (571,375) node [scale=1.44]  {$\mathcal{G}_{1} \cup \mathcal{G}_{2}$};
\end{tikzpicture}
\end{center}
\caption{Some examples of union graph.}
\label{EXAMPLE}
\end{figure}

\newpage

\begin{remark} \label{remremrem}
We observe that for fixed $\equiv_1,\equiv_2$, the union metric graph $\mathcal{G}_1 \cup \mathcal{G}_2$ does depend on the orientations of the edges in $\mathcal E$ (as so do $\mathcal{G}_1, \mathcal{G}_2$, too); this is in sharp contrast to the case of combinatorial graphs.

For instance, we can take the same graphs $\mathcal{G}_1$ and $\mathcal{G}_2$ in the third example in Fig. \ref{EXAMPLE} and just reverse the orientation of one edge as shown in Fig.\ \ref{UNIORI}.
\end{remark}

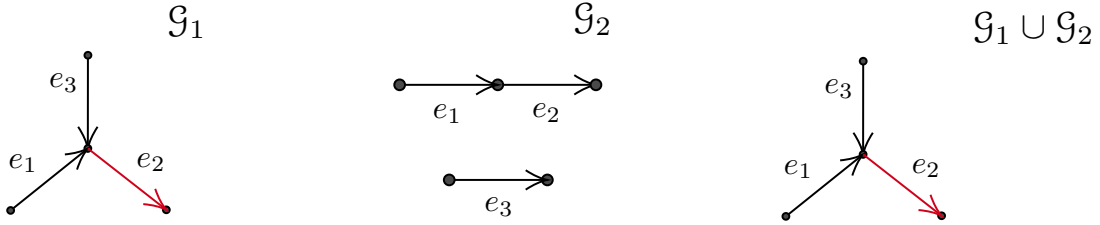
\begin{figure}[h]
\begin{center}
\tikzset{every picture/.style={line width=0.75pt}} %set default line width to 0.75pt        
\begin{tikzpicture}[x=0.75pt,y=0.75pt,yscale=-1,xscale=1]
\draw    (102.16,98.03) -- (102.2,52.99) ;
\draw [shift={(102.16,100.03)}, rotate = 270.05] [color={rgb, 255:red, 0; green, 0; blue, 0 }  ][line width=0.75]    (10.93,-4.9) .. controls (6.95,-2.3) and (3.31,-0.67) .. (0,0) .. controls (3.31,0.67) and (6.95,2.3) .. (10.93,4.9)   ;
\draw  [fill={rgb, 255:red, 74; green, 74; blue, 74 }  ,fill opacity=1 ] (100.48,52.99) .. controls (100.48,52.05) and (101.25,51.3) .. (102.2,51.3) .. controls (103.15,51.3) and (103.92,52.05) .. (103.92,52.99) .. controls (103.92,53.92) and (103.15,54.67) .. (102.2,54.67) .. controls (101.25,54.67) and (100.48,53.92) .. (100.48,52.99) -- cycle ;
\draw  [fill={rgb, 255:red, 74; green, 74; blue, 74 }  ,fill opacity=1 ] (100.44,100.03) .. controls (100.44,99.1) and (101.21,98.35) .. (102.16,98.35) .. controls (103.11,98.35) and (103.87,99.1) .. (103.87,100.03) .. controls (103.87,100.97) and (103.11,101.72) .. (102.16,101.72) .. controls (101.21,101.72) and (100.44,100.97) .. (100.44,100.03) -- cycle ;
\draw    (63.22,131.31) -- (100.6,101.29) ;
\draw [shift={(102.16,100.03)}, rotate = 501.23] [color={rgb, 255:red, 0; green, 0; blue, 0 }  ][line width=0.75]    (10.93,-4.9) .. controls (6.95,-2.3) and (3.31,-0.67) .. (0,0) .. controls (3.31,0.67) and (6.95,2.3) .. (10.93,4.9)   ;
\draw  [fill={rgb, 255:red, 74; green, 74; blue, 74 }  ,fill opacity=1 ] (140.02,131) .. controls (140.02,130.07) and (140.79,129.32) .. (141.73,129.32) .. controls (142.68,129.32) and (143.45,130.07) .. (143.45,131) .. controls (143.45,131.93) and (142.68,132.69) .. (141.73,132.69) .. controls (140.79,132.69) and (140.02,131.93) .. (140.02,131) -- cycle ;
\draw  [fill={rgb, 255:red, 74; green, 74; blue, 74 }  ,fill opacity=1 ] (61.5,131.31) .. controls (61.5,130.38) and (62.27,129.63) .. (63.22,129.63) .. controls (64.16,129.63) and (64.93,130.38) .. (64.93,131.31) .. controls (64.93,132.25) and (64.16,133) .. (63.22,133) .. controls (62.27,133) and (61.5,132.25) .. (61.5,131.31) -- cycle ;
\draw [color={rgb, 255:red, 208; green, 2; blue, 27 }  ,draw opacity=1 ]   (102.16,100.03) -- (140.16,129.77) ;
\draw [shift={(141.73,131)}, rotate = 218.04] [color={rgb, 255:red, 208; green, 2; blue, 27 }  ,draw opacity=1 ][line width=0.75]    (10.93,-4.9) .. controls (6.95,-2.3) and (3.31,-0.67) .. (0,0) .. controls (3.31,0.67) and (6.95,2.3) .. (10.93,4.9)   ;
\draw    (308.92,67.92) -- (356.42,67.92) ;
\draw [shift={(358.42,67.92)}, rotate = 180] [color={rgb, 255:red, 0; green, 0; blue, 0 }  ][line width=0.75]    (10.93,-4.9) .. controls (6.95,-2.3) and (3.31,-0.67) .. (0,0) .. controls (3.31,0.67) and (6.95,2.3) .. (10.93,4.9)   ;
\draw  [fill={rgb, 255:red, 74; green, 74; blue, 74 }  ,fill opacity=1 ] (306.22,67.92) .. controls (306.22,66.43) and (307.43,65.22) .. (308.92,65.22) .. controls (310.41,65.22) and (311.62,66.43) .. (311.62,67.92) .. controls (311.62,69.41) and (310.41,70.62) .. (308.92,70.62) .. controls (307.43,70.62) and (306.22,69.41) .. (306.22,67.92) -- cycle ;
\draw  [fill={rgb, 255:red, 74; green, 74; blue, 74 }  ,fill opacity=1 ] (355.72,67.92) .. controls (355.72,66.43) and (356.93,65.22) .. (358.42,65.22) .. controls (359.91,65.22) and (361.12,66.43) .. (361.12,67.92) .. controls (361.12,69.41) and (359.91,70.62) .. (358.42,70.62) .. controls (356.93,70.62) and (355.72,69.41) .. (355.72,67.92) -- cycle ;
\draw    (259.42,67.92) -- (306.92,67.92) ;
\draw [shift={(308.92,67.92)}, rotate = 180] [color={rgb, 255:red, 0; green, 0; blue, 0 }  ][line width=0.75]    (10.93,-4.9) .. controls (6.95,-2.3) and (3.31,-0.67) .. (0,0) .. controls (3.31,0.67) and (6.95,2.3) .. (10.93,4.9)   ;
\draw  [fill={rgb, 255:red, 74; green, 74; blue, 74 }  ,fill opacity=1 ] (256.72,67.92) .. controls (256.72,66.43) and (257.93,65.22) .. (259.42,65.22) .. controls (260.91,65.22) and (262.12,66.43) .. (262.12,67.92) .. controls (262.12,69.41) and (260.91,70.62) .. (259.42,70.62) .. controls (257.93,70.62) and (256.72,69.41) .. (256.72,67.92) -- cycle ;
\draw  [fill={rgb, 255:red, 74; green, 74; blue, 74 }  ,fill opacity=1 ] (331.22,115.92) .. controls (331.22,114.43) and (332.43,113.22) .. (333.92,113.22) .. controls (335.41,113.22) and (336.62,114.43) .. (336.62,115.92) .. controls (336.62,117.41) and (335.41,118.62) .. (333.92,118.62) .. controls (332.43,118.62) and (331.22,117.41) .. (331.22,115.92) -- cycle ;
\draw    (284.42,115.92) -- (331.92,115.92) ;
\draw [shift={(333.92,115.92)}, rotate = 180] [color={rgb, 255:red, 0; green, 0; blue, 0 }  ][line width=0.75]    (10.93,-4.9) .. controls (6.95,-2.3) and (3.31,-0.67) .. (0,0) .. controls (3.31,0.67) and (6.95,2.3) .. (10.93,4.9)   ;
\draw  [fill={rgb, 255:red, 74; green, 74; blue, 74 }  ,fill opacity=1 ] (281.72,115.92) .. controls (281.72,114.43) and (282.93,113.22) .. (284.42,113.22) .. controls (285.91,113.22) and (287.12,114.43) .. (287.12,115.92) .. controls (287.12,117.41) and (285.91,118.62) .. (284.42,118.62) .. controls (282.93,118.62) and (281.72,117.41) .. (281.72,115.92) -- cycle ;
\draw    (493.16,101.03) -- (493.2,55.99) ;
\draw [shift={(493.16,103.03)}, rotate = 270.05] [color={rgb, 255:red, 0; green, 0; blue, 0 }  ][line width=0.75]    (10.93,-4.9) .. controls (6.95,-2.3) and (3.31,-0.67) .. (0,0) .. controls (3.31,0.67) and (6.95,2.3) .. (10.93,4.9)   ;
\draw  [fill={rgb, 255:red, 74; green, 74; blue, 74 }  ,fill opacity=1 ] (491.48,55.99) .. controls (491.48,55.05) and (492.25,54.3) .. (493.2,54.3) .. controls (494.15,54.3) and (494.92,55.05) .. (494.92,55.99) .. controls (494.92,56.92) and (494.15,57.67) .. (493.2,57.67) .. controls (492.25,57.67) and (491.48,56.92) .. (491.48,55.99) -- cycle ;
\draw  [fill={rgb, 255:red, 74; green, 74; blue, 74 }  ,fill opacity=1 ] (491.44,103.03) .. controls (491.44,102.1) and (492.21,101.35) .. (493.16,101.35) .. controls (494.11,101.35) and (494.87,102.1) .. (494.87,103.03) .. controls (494.87,103.97) and (494.11,104.72) .. (493.16,104.72) .. controls (492.21,104.72) and (491.44,103.97) .. (491.44,103.03) -- cycle ;
\draw    (454.22,134.31) -- (491.6,104.29) ;
\draw [shift={(493.16,103.03)}, rotate = 501.23] [color={rgb, 255:red, 0; green, 0; blue, 0 }  ][line width=0.75]    (10.93,-4.9) .. controls (6.95,-2.3) and (3.31,-0.67) .. (0,0) .. controls (3.31,0.67) and (6.95,2.3) .. (10.93,4.9)   ;
\draw  [fill={rgb, 255:red, 74; green, 74; blue, 74 }  ,fill opacity=1 ] (531.02,134) .. controls (531.02,133.07) and (531.79,132.32) .. (532.73,132.32) .. controls (533.68,132.32) and (534.45,133.07) .. (534.45,134) .. controls (534.45,134.93) and (533.68,135.69) .. (532.73,135.69) .. controls (531.79,135.69) and (531.02,134.93) .. (531.02,134) -- cycle ;
\draw  [fill={rgb, 255:red, 74; green, 74; blue, 74 }  ,fill opacity=1 ] (452.5,134.31) .. controls (452.5,133.38) and (453.27,132.63) .. (454.22,132.63) .. controls (455.16,132.63) and (455.93,133.38) .. (455.93,134.31) .. controls (455.93,135.25) and (455.16,136) .. (454.22,136) .. controls (453.27,136) and (452.5,135.25) .. (452.5,134.31) -- cycle ;
\draw [color={rgb, 255:red, 208; green, 2; blue, 27 }  ,draw opacity=1 ]   (493.16,103.03) -- (531.16,132.77) ;
\draw [shift={(532.73,134)}, rotate = 218.04] [color={rgb, 255:red, 208; green, 2; blue, 27 }  ,draw opacity=1 ][line width=0.75]    (10.93,-4.9) .. controls (6.95,-2.3) and (3.31,-0.67) .. (0,0) .. controls (3.31,0.67) and (6.95,2.3) .. (10.93,4.9)   ;
\draw (152,36) node [scale=1.44]  {$\mathcal{G}_{1}$};
\draw (69.5,107.5) node [scale=1.2]  {$e_{1}$};
\draw (134,107) node [scale=1.2]  {$e_{2}$};
\draw (90,67) node [scale=1.2]  {$e_{3}$};
\draw (356.67,34.67) node [scale=1.44]  {$\mathcal{G}_{2}$};
\draw (283.5,82.5) node [scale=1.2]  {$e_{1}$};
\draw (333.5,82) node [scale=1.2]  {$e_{2}$};
\draw (309.5,129.5) node [scale=1.2]  {$e_{3}$};
\draw (579,39) node [scale=1.44]  {$\mathcal{G}_{1} \cup \mathcal{G}_{2}$};
\draw (460.5,110.5) node [scale=1.2]  {$e_{1}$};
\draw (525,110) node [scale=1.2]  {$e_{2}$};
\draw (481,70) node [scale=1.2]  {$e_{3}$};
\end{tikzpicture}
\end{center}
\caption{If we reverse the orientation of just one edge, the resulting union is different.}
\label{UNIORI}
\end{figure}

Our main result in this section is the following lemma, which characterizes the null space of elliptic operators with natural vertex conditions  associated with the union graph with its connectedness. 

\begin{lemma}\label{lemmacaratt}
Given  $\mathcal{G}_1,\ldots,\mathcal{G}_N$
 metric graphs satisfying the Assumption \ref{a-metgraphs}, let $\cG$ be their union graph 
(see Definition \ref{defunionegrafimetrici}). Let $A_i$ be the elliptic operators associated with $\cG_i$ with natural vertex conditions operators and coefficients $p_i \in W^{1,\infty}(\cG_i)$.
Then
\begin{equation}
\label{characterization}
\mathcal{G} \ \text{is connected} \ \Longleftrightarrow \ \bigcap_{i=1}^{N} \ker A_i = \ \langle \mathbbm{1} \rangle.
\end{equation}
\end{lemma}

Notice that this lemma is remarkably similar to Lemma \ref{Lemmaunione} (which was concerned with combinatorial graphs)
and also their proofs will be similar. 

\begin{proof}
We show the proof for $N=2$, then one can easily extend the result for an arbitrary $N$ by induction.
In general, both $\mathcal{G}_1$ and $\mathcal{G}_2$ have a certain number of disjoint connected components: 
$$\mathcal{G}_1^{(1)},\ldots,\mathcal{G}_1^{(m)} \ \text{ s.t. } \ \mathcal{G}_1=\bigsqcup_{j=1}^{m} \mathcal{G}_1^{(j)},\qquad \text{for some } m \in \mathbb{N}$$
and
$$\mathcal{G}_2^{(1)},\ldots,\mathcal{G}_2^{(n)} \ \text{ s.t. } \ \mathcal{G}_2=\bigsqcup_{k=1}^{n} \mathcal{G}_1^{(k)},\qquad \text{for some } n \in \mathbb{N}.$$
Since the connectedness is just a topological property, notice that the connected components remain the same for every choice of orientation.  

Now assume that $\mathcal{G}$ is connected: we need to show that  $ \ker A_1 \cap \ker A_2 \subseteq \langle \mathbbm{1} \rangle$. 
Thus, we take $f \in \ker A_1 \cap \ker A_2$, in particular from the results in 
Proposition \ref{propositionbasisnull space}
it is well-known that $f$ is constant on each connected component of both $\mathcal{G}_1$ and $\mathcal{G}_2$. 
Take $\xi = (x,\me_h)$ and $\theta = (y,\me_l)$ in $\mathcal{G}$ and without loss of generality we can assume that  $h \ne l$.
Hence, by connectedness of $\mathcal{G}$, there exists a 
chain of adjacent edges $\Gamma_{xy}=\{e_h,e_{i_1},\ldots,e_{i_M},e_l\}$ linking $\xi$ and $\theta$:
$$x \in \me_h \sim \me_{i_1} \sim \cdots \sim \me_{i_M} \sim \me_l \ni y.$$
In particular, edges in $\Gamma_{xy}$ can be incident in $\mathcal{G}_1$ and/or in $\mathcal{G}_2$. 
Thus, taking into account that $f$ is constant on the connected components of both graphs, we deduce that $f$ is constant along $\Gamma_{xy}$ and in particular
$$f(\xi)=f(\theta).$$
Because $\xi, \theta$ are arbitrary, we conclude that $f$ is constant.

In order to prove the opposite implication, we are going to show that if $\mathcal{G}$ is disconnected, then we can find a non constant function such that $f \in \ker A_1 \cap \ker A_2$. 
For simplicity, assume that $\mathcal{G}$ has only two connected components: $\mathcal{G}^{(A)}$ and $\mathcal{G}^{(B)}$. 
Then, both contain a certain number of connected components of $\mathcal{G}_1$ and $\mathcal{G}_2$. In particular, we set
$$J_A=\left\lbrace j\in \{1,\ldots,m\} : \mathcal{G}_1^{(j)} \subseteq \mathcal{G}^{(A)} \right\rbrace, \qquad J_B=\left\lbrace j\in \{1,\ldots,m\} : \mathcal{G}_1^{(j)} \subseteq \mathcal{G}^{(B)} \right\rbrace $$
and
$$K_A=\left\lbrace k\in \{1,\ldots,n\} : \mathcal{G}_2^{(k)} \subseteq \mathcal{G}^{(A)} \right\rbrace, \qquad K_B=\left\lbrace k\in \{1,\ldots,n\} : \mathcal{G}_2^{(k)} \subseteq \mathcal{G}^{(B)} \right\rbrace .$$
Due to the fact that $\mathcal{G}$ is disconnected, it follows that 
$$J_A \cap J_B = \emptyset, \qquad K_A \cap K_B =\emptyset,$$
in fact there cannot exist some connected components of $\mathcal{G}_1$ or $\mathcal{G}_2$ shared by $\mathcal{G}^{(A)}$ and $\mathcal{G}^{(B)}.$ 
This is always true, even if, roughly speaking, we reverse the endpoints of some edge in one of the initial graphs. 
Since the connected components of $\mathcal{G}_1$ and $\mathcal{G}_2$ are invariant under orientation, we will never find an orientation for which some index $j $ is in $J_A \cap J_B$ or some $k $ is in $K_A \cap K_B$. 
Hence, taking the characteristic functions on each connected component such that
$$\ker A_1 = \langle \{\mathbbm{1}_{1,j}\}_{j=1}^{m} \rangle \ \text{ and } \ \ker A_2 = \langle \{\mathbbm{1}_{2,k}\}_{k=1}^{n}\rangle,$$
it is always true that
\begin{equation}
\label{A}
\sum_{j \in J_A} \mathbbm{1}_{1,j} = \mathbbm{1}_A =\sum_{k \in K_A} \mathbbm{1}_{2,k}
\end{equation}
and
\begin{equation}
\label{B}
\sum_{j \in J_B} \mathbbm{1}_{1,j} = \mathbbm{1}_B =\sum_{k \in K_B} \mathbbm{1}_{2,k}.
\end{equation}
At this point, we only need to take some function of the form
$$f = \alpha \mathbbm{1}_A + \beta \mathbbm{1}_B, \qquad \alpha,\beta \in \mathbb{C}$$
and from \eqref{A} and \eqref{B} one gets that $f$ can be written as a linear combination of both bases of $\ker A_i$, $i=1,2$:
$$f= \alpha \sum_{j \in J_A} \mathbbm{1}_{1,j}  + \beta \sum_{j \in J_B} \mathbbm{1}_{1,j} \qquad \Longrightarrow \qquad f \in \ker A_1$$
and
$$f= \alpha \sum_{k \in K_A} \mathbbm{1}_{2,k}  + \beta \sum_{k \in K_B} \mathbbm{1}_{2,k} \qquad \Longrightarrow \qquad f \in \ker A_2.$$
Thus, the proof is complete.
\end{proof}

In the end, we can finally state the following characterization of the asymptotic behavior of the solutions to \eqref{CRE}
in terms of the connectedness of the union graph. The proof is, at this point, a direct consequence of Theorem \ref{t1} and Lemma \ref{lemmacaratt}.

\begin{theorem}\label{FinalTheorem}
Let $(Z(t))_{t\ge0}$ be a semi-Markov process and $\cC$ be a family of graphs that satisfy the Assumptions \ref{hpZ} and \ref{a-metgraphs}, respectively. Then the random propagator $(S(t))_{t\ge0}$ for the Cauchy problem \eqref{CRE} converges in norm $\bP$-almost surely towards the orthogonal projector $P_0$ onto the constants
if and only if the union graph $\cG$ is connected.
\end{theorem}	

\subsection{Randomly switching metric graphs with non-zero second eigenvalue}

As we have previously seen in the combinatorial setting, we are going to 
apply exponential convergence results in the framework of metric graphs in case one (every) graph in $\cC$ is connected.
This follows from an application of Theorem \ref{th-adapt-ADFK} since connectedness of a graph is equivalent to the second eigenvalue being non-zero.

\begin{assumption}\label{ass:metric}
$\cG_1$ is a connected graph.
\end{assumption}

This immediately implies that the union graph $\cG$ is connected, too, and $P_K=P_0$ is the projector onto the constant functions on $\mE$. 
Then, Assumption \ref{a-K} is verified since $\lambda_{2}(A_1)<0$, owing to connectedness of $\cG_1$. 
The setting described here is somehow comparable to the diffusion equation presented in \cite{AreDieKra14}, in the case when their semilinear term is set equal to zero. 
The dependence on time of that model is different from the non-autonomous random evolution problem \eqref{CRE}: while diffusion and conductivity coefficients are in~\cite{AreDieKra14} allowed to vary over time (in a measurable fashion), yielding an operator family $(A(t))_{t\ge 0}$, the evolution is studied on one fixed graph. 
However -- much like in our setting -- the crucial point in~\cite{AreDieKra14} is  that the time average
of the spectral gap of $(A(t))_{t\ge0}$ is bounded above away from zero. 
In their case, this is enforced by assuming that the graph is connected and allows the authors of \cite{AreDieKra14} to prove exponential convergence to equilibrium.
In our setting, the counterpart of~\cite[Thm.~5.4]{AreDieKra14} reads as follows.

\begin{corollary}\label{cor-K-metgraphs}
Let $(Z(t))_{t\ge 0}$ be a semi-Markov process and $\cC$ be a family of graphs that satisfy the Assumptions \ref{hpZ} and \ref{hp3.1}, respectively. Let additionally the Assumption~\ref{ass:metric} hold. 

Then the random propagator $(S(t))_{t \ge 0}$ for the Cauchy problem \eqref{CRE} converges in norm $\bP$-almost surely exponentially fast towards the orthogonal projector $P_0$ with an exponential rate no lower than
\begin{align*}
\alpha = -\sum_{j=1}^N \lambda_2(A_j) \Theta_j
\end{align*} 
that is the average of the eigenvalues  $\lambda_{2}(A_j)$ with respect to the fraction of time $\Theta_j$ spent by the process $Z(t)$ in the various states.
\end{corollary}

\begin{remark}\label{rem:ultimo-metrico}
As in the case of combinatorial graphs discussed in Remark~\ref{rem:unioninters}, we can find in the literature some estimate on the best possible value of the 
parameter $\alpha$.
In this case, we refer e.g.\ to the estimates in \cite[Th\'eo.~3.1]{Nic87}, \cite[Thm.~1]{Fri05}, and~\cite[Thm.~4.2]{KenKurMal16}: 
for a generic connected metric graph $\mathcal G$
\[
-\frac{\pi^2 |\mE|^2}{L^2}\le \lambda_2\le -\frac{\pi^2}{L^2},
\]
where the second inequality is an equality if and only if $\mathcal G$ consists of an interval;
here $|\mE|$ is the number of edges and $L$ is the total length of the graph (the sum of the lengths of the edges).
Therefore, 
the parameter $\alpha$, that is the weighted average of
$-\lambda_2(A_i)$ as $\cG_i$ varies in
 $\cC$, 
 is no lower than $\displaystyle \frac{\pi^2}{L^2}$ 
(as long as the intersection graph $\mathcal G_\cap$ of all graphs in $\cC$ is connected) and no higher than $\displaystyle \frac{\pi^2 |\mE|^2}{L^2}$.
\end{remark}

\end{document}